\documentclass[journal]{IEEEtran}
\usepackage{cite}

\ifCLASSINFOpdf
  \usepackage[pdftex]{graphicx}
  \graphicspath{{./Figs/}}
  \DeclareGraphicsExtensions{.pdf,.png,.epsc}
\else
\fi
\usepackage{amsmath}
\usepackage{amssymb,amsthm}
\usepackage{algpseudocode}
\usepackage{cases}
\usepackage{multirow,makecell}
\usepackage{siunitx}
\newcommand{\norm}[1]{\ensuremath{\left\| #1 \right\|}}

\newcommand{\tr}[1]{\ensuremath{\mathrm{tr}\left( #1 \right)}}
\newcommand{\etr}[1]{\ensuremath{\mathrm{etr}\left( #1 \right)}}
\newcommand{\diag}[1]{\ensuremath{\mathrm{diag}\left( #1 \right)}}
\newcommand{\vect}[1]{\ensuremath{\mathrm{vec}\left( #1 \right)}}
\newcommand{\expect}[1]{\ensuremath{\mathrm{E}\!\left[ #1 \right]}}
\newcommand{\expectbar}[1]{\ensuremath{\bar{\mathrm{E}}\left[ #1 \right]}}
\newcommand{\real}[1]{\ensuremath{\mathbb{R}^{ #1 }}}
\newcommand{\diff}[1]{\ensuremath{\mathrm{d} #1}}
\newcommand{\SO}{\ensuremath{\mathrm{SO}(3)}}
\newcommand{\Sph}{\ensuremath{\mathbb{S}}}
\newcommand{\so}{\ensuremath{\mathfrak{so}(3)}}

\newtheorem{definition}{Definition}
\newtheorem{theorem}{Theorem}
\newtheorem{proposition}{Proposition}
\newtheorem{lemma}{Lemma}
\newtheorem{remark}{Remark}

\makeatletter
\newenvironment{subtheorem}[1]{%
	\def\subtheoremcounter{#1}%
	\refstepcounter{#1}%
	\protected@edef\theparentnumber{\csname the#1\endcsname}%
	\setcounter{parentnumber}{\value{#1}}%
	\setcounter{#1}{0}%
	\expandafter\def\csname the#1\endcsname{\theparentnumber.\Alph{#1}}%
	\ignorespaces
}{%
	\setcounter{\subtheoremcounter}{\value{parentnumber}}%
	\ignorespacesafterend
}
\makeatother
\newcounter{parentnumber}

\newcommand{\algrule}[1][.2pt]{\par\vskip.2\baselineskip\hrule height #1\par\vskip.2\baselineskip}

\begin{document}
%
\title{Matrix Fisher-Gaussian Distribution on $\SO\times\real{n}$ and Bayesian Attitude Estimation}
%
%
%

\author{Weixin Wang, and Taeyoung Lee
\thanks{W. Wang is a doctoral candidate at the Department of Mechanical and Aerospace Engineering,  The George Washington University, Washington DC, USA.
	{\tt\small  wwang442@gwu.edu}}%
\thanks{T. Lee is an Associate Professor at the Department of Mechanical and Aerospace Engineering,  The George Washington University, Washington DC, USA.
	{\tt\small  tylee@gwu.edu}}%
\thanks{This research has been supported in part by NSF under the grant CNS-1837382, and AFOSR under the grant FA9550-18-1-0288.}}

%
%

\markboth{Submitted to transaction on automatic control}%
{Shell \MakeLowercase{\textit{et al.}}: Bare Demo of IEEEtran.cls for IEEE Journals}
%



\maketitle

\begin{abstract}
    In this paper, a new probability distribution, referred to as the matrix Fisher--Gaussian (MFG) distribution, is proposed on the nonlinear manifold $\SO\times\real{n}$.
    It is constructed by conditioning a ($9+n$)-variate Gaussian distribution from the ambient Euclidean space into $\SO\times\real{n}$, while imposing a certain geometric interpretation of the correlation terms to avoid over-parameterization.
    The unique feature is that it may represent large uncertainties in attitudes, linear variables of an arbitrary dimension, and angular--linear correlations between them in a global fashion without singularities associated with local parameterizations. 
    Various stochastic properties and an approximate maximum likelihood estimator of MFG are developed.
    Furthermore, two methods are developed to propagate uncertainties though a stochastic differential equation representing attitude kinematics. 
    Based on these, a Bayesian estimator is proposed to estimate the attitude and time-varying gyro bias concurrently.
    Numerical studies indicate that the proposed estimator exhibits a better accuracy against the well-established multiplicative extended Kalman filter for two challenging cases. 
\end{abstract}


%
\IEEEpeerreviewmaketitle

\section{Introduction}
%
%
%
%

\IEEEPARstart{A}{ttitude} estimation using a gyroscope with time varying biases has been studied since 1960s \cite{crassidis2007survey}.
The unique challenge in attitude estimation is that the attitude evolves on the nonlinear manifold, referred to as the three-dimensional special orthogonal group $\SO$, which cannot be globally identified with the Euclidean space of the same dimension. 
Any three-dimensional parameterization of $\SO$ leads to singularities~\cite{stuelpnagel1964parametrization}.

A milestone in attitude estimation is the development of multiplicative extended Kalman filter (MEKF) \cite{toda1969spars,lefferts1982kalman,markley2003attitude}, where the mean attitude is represented by a quaternion, and the uncertainty distribution about the mean attitude is described by a three-dimensional attitude parameter following a Gaussian distribution. 
MEKF typically relies on the assumption that the uncertainty distribution is highly concentrated about the mean attitude,
and it adopts the framework of extended Kalman filter to estimate the attitude and bias simultaneously through the linearized attitude kinematics equation. 
MEKF has been used as a standard algorithm for attitude estimation in various disciplines, including aerospace engineering~\cite{crassidis2007survey}, robotics \cite{mourikis2007multi,jimenez2010indoor}.
And it has been generalized to matrix Lie groups \cite{barrau2014intrinsic,barrau2016invariant}.
Besides MEKF, a deterministic attitude observer on $\SO$ has been presented in \cite{mahony2008nonlinear}.

Despite its success, MEKF is limited by fundamental issues inherited from local parameterizations and linearization.
First, only the first order term of the error attitude is considered in the linearization of kinematics equation.
Therefore the error attitude has to be small enough for the linearization to be valid.
Second, the Gaussian distribution of local coordinates does not represent attitude uncertainties properly.
For example, suppose that the exponential coordinate is employed.
In this case, the same attitude corresponds to infinitely many different rotation vectors with the difference in length by $2k\pi,\ k\in\mathbb{Z}$.
As such the probability density function used to describe the uncertainty of the rotation vector should be wrapped, i.e., the densities at rotation vectors representing the same attitude need to be added up.
These are ignored in the common implementation of MEKF, which is problematic especially when the attitude uncertainty is widely dispersed and relatively large density values should be wrapped. 
In short, MEKF is not suitable for attitude estimation with large uncertainties. 

In order to overcome the shortcomings of MEKF, there have been efforts to construct attitude estimators with probability distributions defined directly in $\SO$, instead of the Gaussian distribution for local parameterizations of $\SO$.
In \cite{markley2006attitude}, a probability density function on $\SO\times\real{n}$ is inspired by harmonic analysis, and the Fokker-Planck equation describing the evolution of the density is solved in the embedding Euclidean space to construct an attitude filter on $\SO$.

On the other hand, rotational random variables and matrices in compact manifolds have been studied in directional statistics \cite{mardia2009directional}, which have provided various models for distributions. 
In fact, the attitude part of the density function in \cite{markley2006attitude} is exactly the matrix Fisher distribution on $\SO$ \cite{downs1972orientation,khatri1977mises}.
And the matrix Fisher distribution has been shown to be equivalent to the Bingham distribution defined on the unit-sphere for quaternions with antipodal points identified \cite{bingham1974antipodally,prentice1986orientation}.
Utilizing these, attitude estimators have been developed by using the Bingham distribution in \cite{glover2013tracking,kurz2014recursive}, and by using the matrix Fisher distribution in \cite{LeeITAC18}.
However, these are based on probability distributions on $\SO$, and consequently, the gyro bias cannot be estimated concurrently with the attitude. 

To address this, a probability density function on the product space of $\SO\times\real{n}$ should be utilized such that the angular-linear correlation between attitudes and gyro biases can be described. 
The distribution on a simpler state space $\Sph^1\times\real{1}$ where a circular variable is coupled with a linear variable has been studied in \cite{mardia1978model,johnson1978some,kato2008dependent,abe2017tractable}.
In \cite{darling2016uncertainty}, the Bingham distribution on the unit-sphere $\Sph^r$ is extended to the Gauss-Bingham distribution on $\Sph^r\times\real{n}$.
Nevertheless, similar with \cite{markley2006attitude}, the Gauss-Bingham distribution has the following undesirable characteristics. 
First, the angular-linear correlation is imposed by making the matrix parameter defining the Bingham distribution dependent of the linear variable. 
As the Bingham distribution on $\Sph^r$ is defined by an orientation matrix on $\mathrm{SO}(r+1)$ and concentration parameters, the orientation matrix should be parameterized by the linear variables. 
This brings back the issue of local parameterizations, and more importantly, it is challenging to interpret the geometric meaning of the resulting correlation. 
Second, it does not have a closed form maximum likelihood estimator (MLE), which is essential in propagating the uncertainty while assuming the underlying statistical model remains the same.

In this paper, we introduce a new distribution on $\SO\times\real{n}$, referred to as  the matrix Fisher-Gaussian (MFG) distribution.
It is constructed by conditioning a (9+$n$)-variate Gaussian distribution from $\real{9+n}$ into $\SO\times\real{n}$, which leads to the matrix Fisher distribution for the attitude part. 
Then, the correlation term is projected to the tangent space of $\SO$ at the mean attitude. 
The desirable feature is that it provides an intuitive geometric interpretation for the angular-linear correlation between the three-dimensional attitude and $n$-dimensional linear variable. 
This correlation is represented by $3n$ parameters, thereby avoiding a potential over-parameterization in \cite{markley2006attitude} relying on $9n$ parameters. 
Moreover, although the MLE of MFG cannot by solved analytically, there is a closed form approximation by first solving the marginal MLE for the attitude part.
This reduces the computational load significantly compared with \cite{darling2016uncertainty} that requires iterative numerical optimizations. 
Further, MFG can be approximated by a (3+$n$)-variate Gaussian distribution if the attitude part is concentrated.
Therefore, it can be considered as a generalization of the common approach relying on the Gaussian distribution of three-dimensional attitude parameters.

In the second part of this paper, we design an intrinsic form of Bayesian attitude estimator on  $\SO\times\real{n}$ to estimate the attitude and gyro bias concurrently. 
It is composed of two steps: propagation and correction. 
For propagation, the moments of an MFG required for MLE are propagated through the stochastic differential equation for attitude kinematics, from which the propagated MFG is constructed.
There are two approaches for propagation, namely an analytical method suitable for propagation over a short period, and a numerical method based on unscented transform.
Next, for correction, the posterior MFG is constructed according to the Bayes' rule for a given prior density and measurements.
In contrast to MEKF, the proposed approach does not rely on the assumption of small attitude errors, nor does it suffer from the aforementioned wrapping errors. 
Numerical simulations illustrate that the proposed estimator yields a similar performance as MEKF for small initial errors, but it exhibits better accuracy for challenging cases of large uncertainties. 

The remaining of this paper is organized as follows.
In Section II, we present a brief review of the matrix Fisher distribution.
Next, the matrix Fisher-Gaussian distribution is introduced in Section III with various stochastic properties.
A Bayesian attitude estimator based on MFG is developed in Section IV, followed by numerical examples in Section V and conclusions in Section VI.

\section{Matrix Fisher Distribution} \label{sec:MF}

\subsection{Notations and Some Facts}

The three-dimensional orthogonal group $\mathrm{O}(3)$ is defined as
\begin{equation*}
	\mathrm{O}(3) = \{R\in\mathbb{R}^{3\times 3}\big|\ {RR}^T={I}_{3\times 3} \},
\end{equation*}
which is partitioned into two disconnected subsets depending on the sign of the determinant. 
The component that has determinant $+1$ is referred to as the special orthogonal group and it is denoted by $\SO$, which is commonly used to represent the attitude of a rigid body with the right-handed frame. 
Whereas the other component with determinant $-1$ is involved with reflections.
The Lie algebra of $\mathrm{O}(3)$, denoted by $\mathfrak{o}(3)$, is the tangent space of $\mathrm{O}(3)$ at $I_{3\times 3}$, given by
\begin{equation*}
	\mathfrak{o}(3) =  \{ A\in\real{3 \times 3}\, \big|\ A = -A^T \},
\end{equation*}
which is identical to the Lie algebra of $\SO$, or $\so$. 
This is further identified with $(\real{3},\times)$ by the \textit{hat} $\wedge$ map and the \textit{vee} $\vee$ map defined as follows. 
\begin{equation*}
	\mathfrak{o}(3) = \so\ni\begin{bmatrix}
		0 & -\Omega_z & \Omega_y \\
		\Omega_z & 0 & -\Omega_x \\
		-\Omega_y & \Omega_x & 0
	\end{bmatrix} \begin{array}{c}
		\xrightarrow{\text{vee } \vee} \\ \xleftarrow{\text{hat } \wedge}
	\end{array} \begin{bmatrix}
		\Omega_x \\ \Omega_y \\ \Omega_z
	\end{bmatrix}\in\real{3}.
\end{equation*}
The tangent space at an arbitrary $R\in\mathrm{O}(3)$ is denoted by $T_R\mathrm{O}(3)$, and it can be written as:
\begin{equation*}
    T_R\mathrm{O}(3) = \left\{ RA \in\real{3\times 3} \,\big|\ A\in\mathfrak{o}(3) \right\}.
\end{equation*}
If $R\in\SO$, then it is clear that $T_R\mathrm{O}(3) = T_R\SO$ and $T_R\mathrm{O}(3) = T_{-R}\mathrm{O}(3)$.

In this paper, $e_i$ is used to represent the $i$-th standard base vector of $\real{n}$, i.e., the $i$-th column of $I_{n \times n}$ if written in  a vector form.
If not specified otherwise, $I$ is abbreviated for $I_{3 \times 3}$.
Any 3-by-3 diagonal matrix with the diagonal entries of $1$ or $-1$ is denoted by $\mathcal{D}\in\real{3\times 3}$. 
In particular, when it is augmented with subscripts, the subscripts correspond to the diagonal indices for $1$.
For example, $\mathcal{D}_1 = \diag{1,-1,-1}$.
The $n$-dimensional unit sphere is denoted by $\mathbb{S}^n = \{x\in\real{n+1} \big|\ x^Tx=1 \}$.
The trace of a matrix is denoted by $\tr{\cdot}$, the function $e^{\tr{\cdot}}$ is abbreviated as $\etr{\cdot}$.
Operator $\vect{\cdot}$ is used to concatenate the columns of a matrix into a vector, the Kronecker product is denoted by $\otimes$.

In the calculations in this paper, the following equations regarding the trace of a matrix and the hat map will be repeatedly used.
For any $A,B\in\real{3 \times 3}$,
\begin{equation}
	\tr{AB} = \tr{BA} = \tr{B^TA^T}.
\end{equation}
For any $R\in\SO$, $A\in\real{3 \times 3}$ and $x\in\real{3}$,
\begin{gather}
	\widehat{Rx} = R\hat{x}R^T \\
	\hat{x}^2 = xx^T - x^TxI_{3 \times 3} \\
	(\hat{x}A+A^T\hat{x})^\vee = (\tr{A}I_{3 \times 3}-A)x.
\end{gather}

\subsection{Matrix Fisher Distribution on $\SO$}

The matrix Fisher distribution of a random matrix $R\in\SO$ is defined by the following density function
\begin{equation} \label{eqn:MFDensity}
	p(R;F) = \frac{1}{c(F)}\etr{F^TR},
\end{equation}
with respect to the uniform distribution on $\SO$, where $F\in\real{3 \times 3}$ is the parameter describing the shape of the distribution, and $c(F)\in\real{}$ is the normalizing constant.
This is also denoted by $R\sim\mathcal{M}(F)$. 
Various properties of a matrix Fisher distribution can be accessed through the singular value decomposition (SVD) of $F$~\cite{LeeITAC18,khatri1977mises}.
Two conventions of SVD will be discussed, the first one was proposed in \cite{markley1988attitude} and used in \cite{LeeITAC18} to define the matrix Fisher distribution; the second one was used in \cite{prentice1986orientation}.
This paper will focus on the first convention, but we will show the MFG defined through both conventions are actually equivalent.

\begin{subtheorem}{definition} \label{def:psvd}
	\begin{definition} \label{def:psvd1}
		Let the singular value decomposition of $F$ be given by $F= U'S' V'^T$, where $S'\in\real{3\times3}$ is a diagonal matrix composed of the singular values $ s'_1\geq s'_2\geq s'_3\geq 0$ of $F$, and $U',V'\in\mathrm{O}(3)$.
		The \textit{proper singular value decomposition} of $F$ is 
		\begin{align}
			F=USV^T,\label{eqn:USVp}
		\end{align}
		where the rotation matrices $U,V\in\SO$, and the diagonal matrix $S\in\mathbb{R}^{3\times 3}$ are defined as
		\begin{align}
			U &= U'\mathrm{diag}(1,1,\mathrm{det}[U']), \nonumber \\
			S &= \mathrm{diag}(s_1,s_2,s_3)=\mathrm{diag}(s'_1,s'_2,\mathrm{det}[U'V']s'_3), \nonumber \\
			V &= V'\mathrm{diag}(1,1,\mathrm{det}[V']).
		\end{align}
	\end{definition}
	
	\begin{definition} \label{def:psvd2}
		The second convention for proper singular value decomposition of $F$ is
		\begin{align}
			F = U''S''(V'')^T,
		\end{align}
		where the rotation matrices $U'',V''\in\SO$, and the diagonal matrix $S''\in\mathbb{R}^{3\times 3}$ are defined as
		\begin{align}
			U'' &= U'\det(U'), \nonumber \\
			S'' &= S'\det(U')\det(V'), \nonumber \\
			V'' &= V'\det(V').
		\end{align}
	\end{definition}
\end{subtheorem}
The motivation of the above two definitions of the proper SVD is to ensure $U,V\in\SO$, and they are identical if neither or both of the determinants of $U'$ and $V'$ are $-1$.
If only one of them is $-1$, Definition \ref{def:psvd1} only flips the sign of $s'_3$ and the third column of $U'$ or $V'$; whereas Definition \ref{def:psvd2} flips the sign of the entire $S'$, and $U'$ or $V'$.
In the remaining of this subsection, we will mainly focus on Definition \ref{def:psvd1}.

The normalizing constant $c(F)$ only depends on the proper singular values of $F$, i.e.,
\begin{equation}
	c(F) = c(S) = \int_{Q\in\SO}\etr{SQ^T}\diff{Q},
\end{equation}
where $\diff{Q}$ is the bi-invariant Haar measure for $\SO$, normalized such that $\int_{Q\in\SO}\diff{Q} = 1$.
The first order moment of $R$ is given by 
\begin{gather}
	\expect{R} = UDV^T = U\diag{d_1,d_2,d_3}V^T, \\
	d_i = \frac{1}{c(S)} \frac{\partial c(S)}{\partial s_i} \quad \text{for}\ i=1,2,3 \label{eqn:S2D}.
\end{gather}
However, $\expect{R}$ generally does not belong to $\SO$. 
Instead, the \textit{mean attitude} of $R$ is usually interpreted as $UV^T \triangleq M$, which maximizes the density function \eqref{eqn:MFDensity}.

Similar to the Gaussian distribution, the matrix Fisher distribution has three principal axes, given by the columns of $U$ resolved in the standard coordinates of $\real{3}$, or equivalently the columns of $V$ resolved in the coordinates specified by the columns of $M$.
The attitude rotated from the mean attitude $M$ about the $i$-th principal axis for an angle $\theta$, i.e., $R(\theta) = \exp(\theta\widehat{Ue_i})M = M\exp(\theta\widehat{Ve_i})$, has the density
\begin{equation}
	p(R(\theta);F) = \frac{e^{s_i}}{c(S)}\exp((s_j+s_k)\cos\theta),
\end{equation}
where $\{i,j,k\} = \{1,2,3\}$.
This corresponds to the Von Mises distribution \cite{mardia2009directional} for $\theta$ defined on the unit circle $\mathbb{S}^1$, and its concentration around the mean angle $\theta=0$ is specified by $s_j+s_k$.
An interesting property is that when $s_j+s_k$ is sufficiently large, this is approximated by the Gaussian density with the zero mean and the variance $1/(s_j+s_k)$.
This indicates the distribution of $R$ can be approximated by a 3-dimensional Gaussian distribution if $s_j+s_k$ is large for $i=1,2,3$, or in other words, $R$ is concentrated around its mean attitude $M$~\cite{lee2018bayesian}.
The three principal axes can also be defined using $U''$ and $V''$ in Definition \ref{def:psvd2} analogously.

Finally, the matrix Fisher distribution can be constructed by conditioning a Gaussian distribution in $\real{9}$ onto $\SO$ as follows~\cite{downs1972orientation}.
Let $x_R \triangleq \vect{R^T}\in\real{9}$.
If $x_R$ follows a Gaussian distribution with the mean $\vect{(M')^T}\in\real{9}$ and the covariance matrix $I_{3 \times 3} \otimes (K')^{-1}\in\real{9\times 9}$ for $M'\in\mathrm{O}(3)$ and $0\prec K'\in\real{3\times 3}$,
then the distribution of $R$ conditioned by $RR^T = I_{3 \times 3}$ and $\det(R)=1$ is the matrix Fisher distribution with parameter $F=M'K'$.
Furthermore, if $F$ has its SVD as $F=U'S'(V')^T$, then $M'=U'(V')^T$ and $K'=V'S'(V')^T$.

\subsection{Matrix Fisher Distribution on $\mathrm{O}(3)$} \label{sec:MFO(3)}

The matrix Fisher distribution can also be defined on $\mathrm{O}(3)$ with the same density function \eqref{eqn:MFDensity} \cite{downs1972orientation,khatri1977mises}.
Although the distribution on $\mathrm{O}(3)$ is of little interest in practice, it is crucial to understand the symmetry of the matrix Fisher distribution on $\SO$, which is used to interpret some properties of the MFG discussed in subsequent sections.

For a given $F\in\real{3\times 3}$, let its usual singular value decomposition be $F= U'S'(V')^T$ for $U',V'\in\mathrm{O}(3)$ and a diagonal $S'\in\real{3\times 3}$.
For $\theta\in\mathbb{R}$ and $a\in\mathbb{S}^2$, define $R(\theta,a,\mathcal{D})$ to be an orthogonal matrix rotated from $U'\mathcal{D}(V')^T$ about the axis $a$ by the angle $\theta$ as
\begin{align*}
	R(\theta,a,\mathcal{D}) &= U'\exp(\theta\hat{a})\mathcal{D}(V')^T = \exp(\theta\widehat{U'a})U'\mathcal{D}(V')^T \\
	&= U'\mathcal{D}(V')^T\exp(\theta\widehat{V'\mathcal{D}a}).
\end{align*}
Recall that $\mathcal{D}$ is a diagonal matrix with its diagonal elements composed of $\pm 1$ as introduced in Section II.A.
Also, note that $U'(V')^T$ does not necessarily belong to $\SO$.
Then,
\begin{equation} \label{eqn:O3Density}
	\tr{F^TR} = \begin{bmatrix} s'_1 & s'_2 & s'_3 \end{bmatrix} \mathcal{D} \begin{bmatrix} a_1^2+(a_2^2+a_3^2)\cos\theta \\ a_2^2+(a_1^2+a_3^2)\cos\theta \\ a_3^2+(a_1^2+a_2^2)\cos\theta \end{bmatrix}.
\end{equation}

This equation can be used to derive the global and local extremes of $p(R;F)$.
Here the term ``local'' means the density is minimized or maximized only in one component of $\mathrm{O}(3)$, i.e., in $\SO$ or $\mathrm{O}(3)\backslash\SO$, but not globally in $\mathrm{O}(3)$.
Because $s'_1, s'_2, s'_3$ are non-negative, $\tr{F^TR}$ reaches its global maximum of $s'_1 + s'_2 + s'_3$ at $R(0,a,I_{3 \times 3}) = U'(V')$, and reaches its global minimum of $-s'_1-s'_2-s'_3$ at $R(0,a,-I_{3\times 3}) = -U'(V')^T$.
Note that the maximum and minimum locates in different connected components of $\mathrm{O}(3)$.
On the other hand, the local minimum of $\tr{F^TR}$ in the same component of $U'(V')$ is reached at $R(0,a,\mathcal{D}_3) = U'\mathcal{D}_3(V')^T$; and the local maximum of the other component where $U'(V')^T$ does not belong is reached at $R(0,a,-\mathcal{D}_3) = -U'\mathcal{D}_3(V')$.


Now we examine the two conventions of proper SVD in Definition \ref{def:psvd} when $U'$ or $V'$ has the determinant of $-1$.
Following Definition \ref{def:psvd1}, $UV^T = -U'\mathcal{D}_3(V')$ gives the \textit{local} maximum density in $\SO$.
The concentrations along the principal axes are determined by $s'_2-s'_3$, $s'_1-s'_3$ and $s'_1+s'_2$ respectively, which are the same as those at $-UV^T$, i.e. the \textit{local} minimum in $\mathrm{O}(3)\setminus\SO$.
Following Definition \ref{def:psvd2}, $U''(V'')^T = -U'(V')^T$ gives the \textit{global} minimum density of $\mathrm{O}(3)$.
The concentrations along the principal axes are determined by $s'_2+s'_3$, $s'_1+s'_3$ and $s'_1+s'_2$ respectively, which are the same as those at $U'(V')^T$, i.e. the \textit{global} maximum of $\mathrm{O}(3)$.
From this perspective, it is more natural to use Definition \ref{def:psvd2}, since the principal axes are expressed as rotations from the \textit{global} minimum, which shares the same tangent space as the \textit{global} maximum.
However, in engineering applications, we are more interested in the local maximum on $\SO$, which gives the best estimate of the attitude.


\section{Matrix Fisher-Gaussian Distribution}

In this section, we first give the formal definition of the matrix Fisher-Gaussian distribution, then we present a geometric interpretation and develop some properties as well as an approximate MLE.

\begin{subtheorem}{definition} \label{def:MFG}
	\begin{definition} \label{def:MFG1}
		The random variables $(R,x)\in\mathrm{SO}(3)\times\mathbb{R}^n$ follow the matrix Fisher--Gaussian distribution with parameters ${\mu}\in\mathbb{R}^n$, ${\Sigma}=\Sigma^T\in\mathbb{R}^{n\times n}$, $U,V\in\mathrm{SO}(3)$, $S=\mathrm{diag}(s_1,s_2,s_3)\in\mathbb{R}^{3\times 3}$ with $s_1 \geq s_2 \geq |s_3| \geq 0$ and ${P}\in\mathbb{R}^{n\times 3}$, if it has the following density function:
		\begin{align} \label{eqn:MFGDensity}
			p(R,&x;\mu,\Sigma,V,S,U,{P}) = \frac{1}{c(S)\sqrt{(2\pi)^n\mathrm{det}({\Sigma}_c)}}\times \nonumber \\ &\exp\left\{-\frac{1}{2}(x-\mu_c)^T\Sigma_c^{-1}(x-\mu_c)\right\} \mathrm{etr}\left\{FR^T\right\},
		\end{align}
		where $\mu_c\in\real{n}$ is given by
		\begin{equation} \label{eqn:MFGMiuc}
			\mu_c = \mu+P\nu_R,
		\end{equation}
		with 
		\begin{equation} \label{eqn:vR}
			\nu_R = ({Q}S-S{Q}^T)^\vee,
		\end{equation}
		for ${Q}=U^TRV$, and $0\prec {\Sigma}_c\in\real{n \times n}$  is defined as
		\begin{equation} \label{eqn:MFGSigmac}
			\Sigma_c = \Sigma-P(\tr{S}I_{3\times 3}-S)P^T,
		\end{equation}
		Also, $F=USV^T\in\real{3 \times 3}$, and $c(S)$ is the normalizing constant of the corresponding matrix Fisher distribution.
		This distribution is denoted as $\mathcal{MG}({\mu},{\Sigma},{P},U,S,V)$.
	\end{definition}

	\begin{definition} \label{def:MFG2}
		$\mathcal{MG}({\mu},{\Sigma},{P},U,S,V)$ is defined by replacing the constraint on $S$ in Definition \ref{def:MFG1} by $s_1 \geq s_2 \geq s_3 \geq 0$ or $s_1 \leq s_2 \leq s_3 \leq 0$.
	\end{definition}
\end{subtheorem}

The two definitions only differ when $s_3<0$.
When Definition \ref{def:MFG1} is used, $USV^T = F$ is the proper SVD following Definition \ref{def:psvd1}; whereas when Definition \ref{def:MFG2} is used, $USV^T=F$ is the proper SVD following Definition \ref{def:psvd2}.

The probability density function of MFG given by \eqref{eqn:MFGDensity} is composed of three terms:
the first one is for normalization; the second term is for $x$ and it has the form as $\mathcal{N}(\mu_c, \Sigma_c)$; the last term is for $R$ and it is identical to the matrix Fisher Distribution. 
From its definition, it is straightforward to see that the marginal distribution of $R$ is a matrix Fisher distribution with parameter $F$, and
the distribution of $x$ conditioned by $R$ is Gaussian with $x|R\sim\mathcal{N}(\mu_c(R), \Sigma_c)$. 

When $P=0$, we have $x\sim\mathcal{N}(\mu,\Sigma)$ and $R\sim\mathcal{M}(F)$, and they are independent. 
As such, the attitude-linear correlation is specified by the $3n$ element matrix $P$.
More specifically, the correlation between $x$ and $R$ is from the fact that $(\mu_c,\Sigma_c)$ is dependent of $\mu$, $R$, $F$, and $P$. 
When conditioned by $R$, the mean is shifted by $P\nu_R$ in \eqref{eqn:MFGMiuc}, where $\nu_R$ in \eqref{eqn:vR} indicates how $R$ deviates from the mean attitude $U^TV$. 
For example, when $R=UV^T$ or equivalently $Q=I$, $\nu_R=0$ and $\mu=\mu_c$.

\subsection{Geometric Construction and Interpretation}

The construction of MFG is motivated by \cite{mardia1978model}, where a distribution on the cylinder is constructed by conditioning a 3-dimensional Gaussian distribution on $\real{3}$ into $\mathbb{S}^1 \times \real{1}$.
However, in \cite{mardia1978model} two parameters are used to quantify the angular-linear correlation, despite both of the angular part $\mathbb{S}^1$ and the linear part $\real{1}$ are one-dimensional.
This is because the correlation is formulated between the embedding space of $\mathbb{S}^1$ and $\real{1}$, namely $\real{2}$ and $\real{1}$.

\begin{figure}
	\centering
	\includegraphics{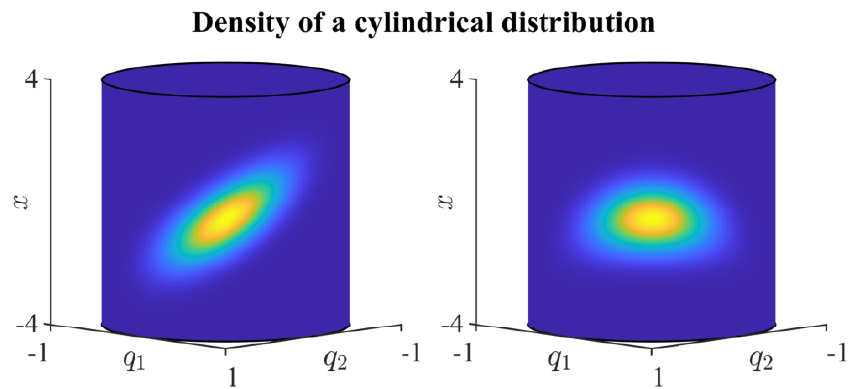}
    \caption{The illustration of the cylindrical distribution with the mean angle $\pi/4$~\cite{mardia1978model}.
        In the left figure, the correlation between the linear variable $x\in\real{1}$ and the angular variable $q=(q_1,q_2)\in\Sph^1$ is only nonzero along the tangent direction of $\mathbb{S}^1$ at $\pi/4$, and the distribution of $q$ conditioned by $x$ varies as $x$ is altered, and vice versa. 
        However in the right figure,  the correlation is specified along the radial direction of $q$, and there is no clear correlation between $x$ and $q$.
        These illustrate the correlation along the tangent direction captures what we expect as a \textit{linear} correlation between two random variables.
	\label{fig:AngularLinearCorrelation}}
\end{figure}

The correlation between two random variables describes how one random variable is varied from its mean, when the other random variable is perturbed from its own mean. 
It shows only the linear relationship averaged. 
As such, we wish that the correlation between $\Sph^1$ and $\real{1}$ is described by a single parameter. 
The key observation we made is that the unit-vector in $\Sph^1$ cannot be varied from the mean exclusively along the radial direction due to the unit-length constraint. 
Further, when it is rotated, the variation along the radial direction is constrained by the variation along the tangential direction. 
Thus, the first-order correlation between $\Sph^1$ and $\real{1}$ can be described by the correlation between the linear variable and the \textit{tangential} direction of the angular variable at the mean angle.
This is illustrated in Fig. \ref{fig:AngularLinearCorrelation}.

Similarly, if the correlation between $R\in\SO$ and $x\in\real{n}$ is defined as in~\cite{mardia1978model}, we would need $9n$ parameters. 
Instead, the embedding space of $\SO$, namely $\real{3\times 3}$, is decomposed into two parts, namely the tangent space at the mean attitude and its orthogonal complement.
And the correlation with linear variables along the orthogonal complement is set to zero. 
In other words, the correlation is defined between the tangent space and the linear variable, thereby reducing the number of free parameters into $3n$.
This yields the following geometric construction of MFG introduced in Definition \ref{def:MFG}.

\begin{theorem} \label{thm:conditioning}
    Consider the parameters $(\mu, \Sigma, V, S, U, P)$ defining MFG as introduced in Definition \ref{def:MFG}.
	Let $M = UV^T\in\SO$ and $K = VSV^T\in\real{3\times 3}$, ${\mu}_R = \mathrm{vec}(M^T)\in\real{9}$, and ${\Sigma}_R^{-1} = I_{3\times 3} \otimes K \in\real{9\times 9}$.

    Also let ${t}_i = \mathrm{vec}[(M\widehat{Ve_{i}})^T]$ for $i\in\{1,2,3\}$ be the basis for the tangent space of $\SO$ at $M$  embedded in $\real{9}$.
    And let $\{t_4,\ldots,t_9\}$ be the orthogonal complement of $\{t_1,t_2,t_3\}$ in $\mathbb{R}^9$.
	Define ${T} = [{t}_1,\ldots,{t}_9]^T\in\real{9\times 9}$, and $P_R = [P,{0}_{n\times6}]{T}\in\real{n\times 9}$. 
	
	Suppose $(x_R,x)\in\real{9+n}$ follows the Gaussian distribution with 
    \begin{align}
        \begin{bmatrix} x_R \\ x \end{bmatrix} 
        \sim \mathcal{N} \left(
            \begin{bmatrix} \mu_R \\ \mu \end{bmatrix},
            \begin{bmatrix}
                \Sigma_R & P_R^T \\
                P_R & \Sigma
        \end{bmatrix} \right).\label{eqn:xRx_Gaussian}
    \end{align}
	Then for $R=\mathrm{vec}^{-1}(x_R)^T\in\real{3\times 3}$, the distribution of $(R,x)$ conditioned on $R^TR = I_{3\times3}$ and $\mathrm{det}(R)=1$ is $\mathcal{MG}({\mu},{\Sigma},{P},U,S,V)$ following either Definition \ref{def:MFG1} or \ref{def:MFG2} depending on the signs of $S$.
\end{theorem}

\begin{proof}
	The joint density of $(x_R,x)$ can be written in the form of conditional-marginal density as
	\begin{align} \label{eqn:BlockGaussian}
		p(x_R,x) = \frac{1}{c}&\mathrm{exp}\left\{-\frac{1}{2}(x-{\mu}_c)^T{\Sigma}_c^{-1}(x-{\mu}_c)\right\} \nonumber \\
		\times &\mathrm{exp}\left\{-\frac{1}{2}(x_R-{\mu}_R)^T\Sigma_R^{-1}(x_R-{\mu}_R)\right\},
	\end{align}
	where $c$ is the normalizing constant,  ${\mu}_c = {\mu}+P_R\Sigma_R^{-1}(x_R-{\mu}_R)\in\real{n}$, and ${\Sigma}_c = {\Sigma}-P_R\Sigma_R^{-1}P_R^T\in\real{n\times n}$.
	
	The exponent of the last term of \eqref{eqn:BlockGaussian} can be written as
	\begin{align*}
		&-\frac{1}{2}(x_R-\mu_R)^T\Sigma_R^{-1}(x_R-\mu_R) \\
		= &-\frac{1}{2}\tr{K(\mathrm{vec}^{-1}(x_R)^T-M)^T(\mathrm{vec}^{-1}(x_R)^T-M)} \\
		= &\ \tr{MK\mathrm{vec}^{-1}(x_R)} + C = \tr{FR^T} + C,
	\end{align*}
	where the $C$ is a constant independent of $x_R$ or $x$ because $\mathrm{vec}^{-1}(x_R)\mathrm{vec}^{-1}(x_R)^T = I_{3 \times 3}$ when conditioned on $R^TR = I_{3 \times 3}$.
	So the second term on the right hand side of \eqref{eqn:BlockGaussian} reduces to a matrix Fisher density after conditioning.
	
	Next, for the first term on the right hand side of \eqref{eqn:BlockGaussian}, the second part of $\Sigma_c$ is
	\begin{equation*}
		P_R\Sigma_R^{-1}P_R^T = {P}[{t}_1,{t}_2,{t}_3]^T\Sigma_R^{-1}[{t}_1,{t}_2,{t}_3]{P}^T 
		\triangleq {P} \tilde{\Sigma}_R^{-1} {P}^T,
	\end{equation*}
	for some $\tilde{\Sigma}_R^{-1} \in \real{3 \times 3}$.
	Let $t_i\in\real{9}$ be equally split into three vectors $\{t_{i1},t_{i2},t_{i3}\}$, then the $i,j$-th entry of $\tilde{\Sigma}_R^{-1}$ can be written as
	\begin{align*}
		(\tilde{\Sigma}_R^{-1})_{ij} &= \sum_{k=1}^3{t}_{ik}^T K {t}_{jk} = \mathrm{tr}\left([{t}_{i1},{t}_{i2},{t}_{i3}]^T K [{t}_{j1},{t}_{j2},{t}_{j3}]\right) \nonumber \\
		&= \mathrm{tr}( M \widehat{Ve_i} K \widehat{Ve_j}^T M ^T)
		= \mathrm{tr}(S\hat{e}_j^T\hat{e}_i),
	\end{align*}
	which implies $\tilde{\Sigma}_R^{-1} = \tr{S}I_{3 \times 3}-S$, thus $\Sigma_c$ has the same expression as in \eqref{eqn:MFGSigmac}.
	Besides, the second part of $\mu_c$ is
	\begin{equation*}
		P_R\Sigma_R^{-1}(x_R-{\mu}_R) = {P}[{t}_1,{t}_2,{t}_3]^T\Sigma_R^{-1}\mathrm{vec}(R^T-M^T),
	\end{equation*}
	and
	\begin{gather*}
		t_i^T\Sigma_R^{-1} \mathrm{vec}(R^T-M^T)
		= \mathrm{tr}( M \widehat{Ve_i}K(R^T-M^T)) \nonumber \\
		= \mathrm{tr}(S{Q}^T\hat{e}_i-S\hat{e}_i) = e_i^T({Q}S-S{Q}^T)^\vee, 
	\end{gather*}
	where ${Q} = U^TRV$.
	This shows $\mu_c$ also has the same expression as in \eqref{eqn:MFGMiuc}.
	In conclusion, the density function \eqref{eqn:BlockGaussian} is the same as \eqref{eqn:MFGDensity} after conditioning on $RR^T = I_{3 \times 3}$ and $\det(R) = 1$.
\end{proof}

In \eqref{eqn:xRx_Gaussian}, the covariance between $x_R=\mathrm{vec}(R^T)$ and $x$ is $P_R =[ P, 0_{n\times 6}]T$, i.e., it is non-zero along the first three columns of $T$, which is a basis for the tangent space of $\SO$ at $UV^T$.
The basis is chosen as the principal axes of the matrix Fisher part, so that the correlation $P$ is expressed with respect to the principal axes, which simplifies the corresponding mathematical analysis and provides geometric interpretations.
Also, there are only $3n$ parameters required to quantify the linear correlation between 3-dimensional attitudes and $n$-dimensional linear variables, instead of $9n$.

\begin{figure}
	\centering
	\includegraphics{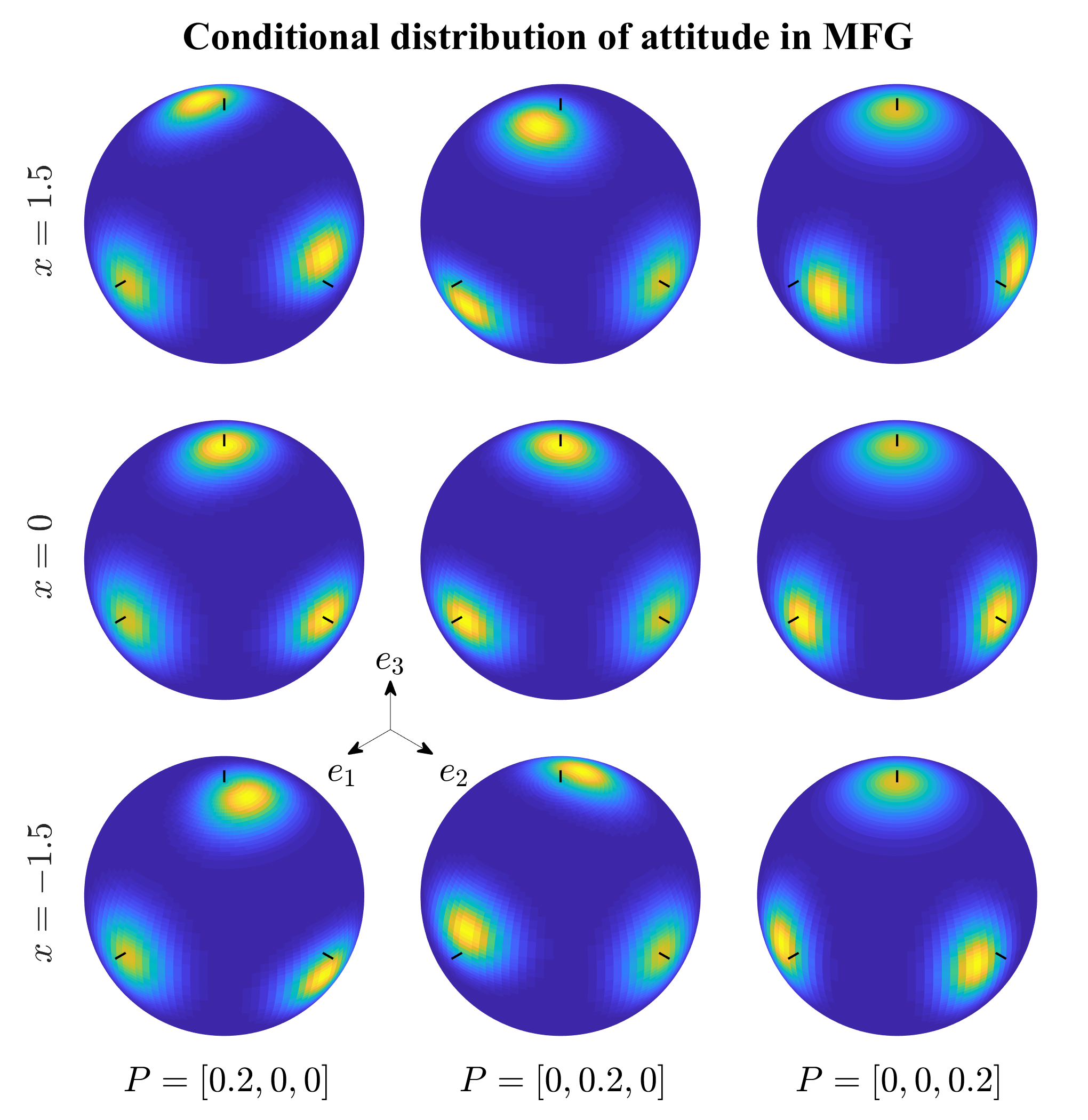}
	\caption{Visualization of the attitude-linear correlation for $(R,x)\in\SO\times\real{1}$: the density for $R$ conditioned on ${x}$ is illustrated on the unit-sphere for three correlation matrices $P$. 
    More specifically, the marginal distribution for each column of $R$ is shown on the unit sphere. 
    The parameters are $n=1$, $\mu=0$, $\Sigma=1$, $U=V=I_{3\times3}$, $S=10I_{3\times3}$.
	For each  $P$, three conditioning values of $x$ are considered. 
	The first column is for ${P}=[0.2,0,0]$.
	When $x$ is increased from $-1.5$ to $1.5$, the conditional distribution for $R$ is rotated about the $e_1$ axis. 
	Similarly, in the second column for ${P}=[0,0.2,0]$, the variation of $x$ is correlated with rotating the distribution of $R$ along the $e_2$ axis.
	The third column shows rotations about $e_3$ due to the correlation. 
	\label{fig:MFGCorrelation}}
\end{figure}

Note that if using Definition \ref{def:MFG1}, when $s_3 < 0$, $UV^T$ is the mean attitude, i.e. the attitude that maximizes the density in $\SO$, and the proper SVD of $F$ uses the convention in Definition \ref{def:psvd1};
if using Definition \ref{def:MFG2}, when $s_3 < 0$, $UV^T$ is the attitude that minimizes the density in $\mathrm{O}(3)$ globally, but $UV^T$ has the same tangent space as $-UV^T$ which maximizes the density in $\mathrm{O}(3)$, and the proper SVD of $F$ uses the convention in Definition \ref{def:psvd2}.
In other words, Definition $\ref{def:MFG1}$ and $\ref{def:MFG2}$ constrain the correlation between $x_R$ and $x$ to be nonzero in different tangent spaces of $\SO$.
Also, in both cases \eqref{eqn:BlockGaussian} is not a true Gaussian density, since $\Sigma_R$ is not positive definite.
However, due to the compactness of $\SO$, the integral of the second term over $\SO$ is always finite, this does not affect the validity of MFG after conditioning.

Because of this construction, the correlation $P$ has a clear geometric interpretation:
if $P_{ij} > 0$, as $x_i$ becomes increased, the distribution of $R$ conditioned on $x$ rotates about the $j$-th principal axis of the matrix Fisher part, and vice versa.
See Fig. \ref{fig:MFGCorrelation} for a simple example.

\subsection{Equivalent Distributions}

As seen in Theorem \ref{thm:conditioning} and \eqref{eqn:vR}, the definition of MFG relies on the proper SVD of $F$ as the correlation is defined along the principal axes represented by $V$.
There are two uniqueness issues associated with the definition of proper SVD.
The trivial one is that when $U$ and $V$ undergo simultaneous sign changes of two columns, they are still the proper left and right singular vectors of $F$.
The other case is more interesting: when $F$ has repeated singular values, the corresponding proper singular vectors in $U$ and $V$ are only unique up to a rotation.
This means the same MFG can be potentially parameterized differently with different proper SVDs.
However, as the next proposition shows, an MFG can be uniquely parameterized by the intermediate parameters $F$, $\mu_c$ and $\Sigma_c$.

\begin{proposition} \label{prop:MFGEquivalent}
	Two matrix Fisher-Gaussian distributions, namely $\mathcal{MG}(\mu,\Sigma,P,U,S,V)$ and $\mathcal{MG}(\tilde{\mu},\tilde{\Sigma},\tilde{P},\tilde{U},\tilde{S},\tilde{V})$ are equivalent if and only if $F=\tilde{F}$, $\mu_c = \tilde{\mu}_c$ for all $R\in\SO$, and $\Sigma_c = \tilde{\Sigma}_c$.
\end{proposition}

The proof of this proposition requires the following lemma, saying that a matrix Fisher distribution is uniquely parameterized by its parameter $F$.

\begin{lemma} \label{lemma:MFEquivalent}
	Two matrix Fisher distributions, namely  $\mathcal{M}(F)$ and $\mathcal{M}(\tilde F)$ are equivalent if and only if $F=\tilde{F}$.
\end{lemma}
\begin{proof}
	It is trivial to show $F = \tilde F$ implies $\mathcal{M}(F)=\mathcal{M}(\tilde F)$.
	Next, suppose $\mathcal{M}(F)=\mathcal{M}(\tilde F)$, i.e.,  $\frac{\etr{FR^T}}{c(F)} = \frac{\etr{\tilde{F}R^T}}{c(\tilde{F})}$ for all $R\in\SO$.
	Let $\Delta F = F-\tilde{F}$ and $\Delta c = \log(c(F)/c(\tilde{F}))$, the above equation is equivalent to
	\begin{equation} \label{eqn:MFEquivalent}
		\tr{\Delta FR^T} - \Delta c = 0.  
	\end{equation}
	Substitute $R = I_{3\times3}$, $\mathcal{D}_1$, $\mathcal{D}_2$ and $\mathcal{D}_3$ into \eqref{eqn:MFEquivalent} to obtain
	\begin{align*} 
		\Delta F_{11} + \Delta F_{22} + \Delta F_{33} - \Delta c &= 0, \\
		\Delta F_{11} - \Delta F_{22} - \Delta F_{33} - \Delta c &= 0, \\
		-\Delta F_{11} + \Delta F_{22} - \Delta F_{33} - \Delta c &= 0, \\
		-\Delta F_{11} - \Delta F_{22} + \Delta F_{33} - \Delta c &= 0,
	\end{align*}
	which shows $\Delta F_{11} = \Delta F_{22} = \Delta F_{33} = \Delta c = 0$.
	Next, substituting $R = \begin{bmatrix} 0 & 1 & 0 \\ 1 & 0 & 0 \\ 0 & 0 & -1 \end{bmatrix}$ and $\begin{bmatrix} 0 & -1 & 0 \\ 1 & 0 & 0 \\ 0 & 0 & 1 \end{bmatrix}$ into \eqref{eqn:MFEquivalent} yields $\Delta F_{21} = 0$.
	Similarly, other entries of $\Delta F$ can be shown to be zeros. 
	Therefore, $F=\tilde{F}$.
\end{proof}

Next, we present the proof for Proposition \ref{prop:MFGEquivalent}.

\begin{proof}[Proof for Proposition \ref{prop:MFGEquivalent}]
	The sufficiency directly follows from \eqref{eqn:MFGDensity} since the density function is determined by $F$, $\mu_c$ and $\Sigma_c$.
	Next we show necessity.
	Let $p$ and $\tilde p$ be the density functions for the two sets of parameters respectively.
	Suppose $p(x,R) = \tilde p(x, R)$  for all $x\in\real{n}$ and $R\in\SO$.
	Since $\SO$ is compact, and $\mu_c$ is continuous in $R$ as seen from \eqref{eqn:MFGMiuc}, $\norm{\mu_c}$ has an upper bound.
	Therefore,
	\begin{equation*}
		\lim\limits_{\norm{x}\to\infty} p(x,R) = \frac{\etr{FR^T}}{c(F)\sqrt{(2\pi)^n\mathrm{det}({\Sigma}_c)}}.
	\end{equation*}
	Because $p(x,R) = \tilde p(x,R)$ for all $x\in\real{n}$, the above equation implies $\frac{\etr{FR^T}}{c(F)\sqrt{(2\pi)^n\det(\Sigma_c)}} = \frac{\etr{\tilde{F}R^T}}{c(\tilde{F})\sqrt{(2\pi)^n\det(\tilde{\Sigma}_c)}}$, and by following the same argument in Lemma \ref{lemma:MFEquivalent}, $F=\tilde{F}$.
	
	Next, since $F=\tilde{F}$ and $p(x,R) = \tilde p(x,R)$, we have $(x-\mu_c)^T\Sigma_c^{-1}(x-\mu_c) = (x-\tilde{\mu}_c)^T\tilde{\Sigma}_c^{-1}(x-\tilde{\mu}_c)$ for all $x\in\real{n}$ and $R\in\SO$.
	Substituting $x = \mu_c$ yields $\mu_c = \tilde{\mu}_c$ for all $R\in\SO$, since $\tilde{\Sigma}_c$ is positive-definite.
	Also, substituting $x = \mu_c+e_i$ shows $(\Sigma_c^{-1})_{ii} = (\tilde{\Sigma}_c^{-1})_{ii}$, and similarly, substituting $x = \mu_c+e_i+e_j$ for $i \neq j$ yields $(\Sigma_c^{-1})_{ij} = (\tilde{\Sigma}_c^{-1})_{ij}$ since $\Sigma_c^{-1}$ and $\tilde{\Sigma}_c^{-1}$ are symmetric, where $e_i$ is the $i$-th column of $I_{n \times n}$.
	Therefore $\Sigma_c = \tilde{\Sigma}_c$.
\end{proof}

To further break down the equivalent conditions for MFG into the parameters $U$, $S$ and $V$, a detailed analysis on the multiplicity of singular values $S$ is required.
In Appendix \ref{appendix:MFGEquivalent}, we provide a complete characterization of the equivalent parameters of an MFG.
It shows that an MFG can be parameterized differently by rotating $U$, $V$, and $P$ in a consistent way, if $S$ has repeated values.

\subsection{Equivalence of the Two Definitions}


As discussed in Section \ref{sec:MFO(3)} and Theorem \ref{thm:conditioning}, it appears that Definition \ref{def:MFG2} is more natural for an MFG, because it uses the tangent space at the global maximum of $\mathrm{O}(3)$, instead of the local maximum of $\SO$ to define the correlation.
Also, Theorem \ref{thm:MFGEquivalent} shows under this definition, the covariance matrix $\Sigma$ of the linear variables dose not change if different sets of parameters are used, as opposed to case \ref{case:s1=s2=-s3 neq 0}) and \ref{case:s1 neq s2=-s3 neq 0}) using Definition \ref{def:MFG1}.
However, as the next theorem will show, the two definitions are actually equivalent in the sense that for every MFG using Definition \ref{def:MFG1}, there is an equivalent MFG using Definition \ref{def:MFG2}, and vice versa.

\begin{theorem} \label{thm:defEquivalent}
	Let $S=\diag{s_1,s_2,s_3}$ with $s_1 \geq s_2 \geq -s_3>0$, then $\mathcal{MG}(\mu,\Sigma,P,U,S,V)$ using Definition \ref{def:MFG1} is equivalent to $\mathcal{MG}(\mu,\tilde{\Sigma},P,U,S\mathcal{D}_3,V\mathcal{D}_3)$ using Definition \ref{def:MFG2}, where $\tilde{\Sigma} = \Sigma - P(\tr{S}I-S)P^T + P(\tr{S\mathcal{D}_3}I-S\mathcal{D}_3)P^T$.
\end{theorem}
\begin{proof}
	By Proposition \ref{prop:MFGEquivalent}, it suffices to prove $F=\tilde{F}$, $\Sigma_c=\tilde{\Sigma}_c$, $\mu_c=\tilde{\mu}_c$ for all $R\in\SO$, where $F,\,\Sigma_c,\,\mu_c$ are the intermediate parameters for the MFG using definition \ref{def:MFG1}, and $\tilde{F},\tilde{\Sigma}_c,\tilde{\mu}_c$ are those using definition \ref{def:MFG2}.
	This is immediate since
	\begin{align*}
		&\tilde{F} = US\mathcal{D}_3\mathcal{D}^T_3V^T = USV^T = F, \\
		&\tilde{\Sigma}_c = \tilde{\Sigma} - P(\tr{S\mathcal{D}_3}I-S\mathcal{D}_3)P^T = \Sigma_c,
	\end{align*}
	and for all $R\in\SO$,
	\begin{align*}
		\tilde{\mu}_c(R) &= \mu + P(U^TRV\mathcal{D}_3S\mathcal{D}_3-\mathcal{D}_3S\mathcal{D}_3V^TR^TU)^\vee \\
		&= \mu + P(U^TRVS-SV^TR^TU)^\vee = \mu_c(R),
	\end{align*}
	which finishes the proof.
\end{proof}

Since when $s_3 \geq 0$, the two definitions are the same, Theorem \ref{thm:defEquivalent} proves there is a one-to-one correspondence between the two definitions.
Therefore, for any theorem regarding one definition, there is a counterpart theorem formulated for the other definition by twisting the parameters properly.
Due to this equivalence relationship, we only consider Definition \ref{def:MFG1} throughout the remaining parts of this paper for its convenience in engineering applications.

\subsection{Gaussian Approximation}

It has been shown that a matrix Fisher distribution can be approximated by a 3-dimensional Gaussian distribution when it is highly concentrated, or equivalently $s_3 \gg 0$~\cite{lee2018bayesian}. 
This is also true for MFG as stated in the following theorem.
\begin{theorem}
	Suppose $(x,R) \sim \mathcal{MG}(\mu,\Sigma,P,U,S,V)$.
    Let $R=U\exp(\hat{\eta})V^T$.
	If $s_3 \gg 0$, then $(x,\eta)$ approximately follows a $3+n$ dimensional Gaussian distribution with 
    \begin{align}
        \begin{bmatrix} x \\ \eta \end{bmatrix}
        \sim
        \mathcal{N} \left(
            \begin{bmatrix} \mu \\ 0 \end{bmatrix},
            \begin{bmatrix} \Sigma & P \\ P^T & (\tr{S}I_{3 \times 3}-S)^{-1} \end{bmatrix}
        \right).\label{eqn:MFGapprox}
    \end{align}
\end{theorem}
\begin{proof}
	For the matrix Fisher density part in \eqref{eqn:MFGDensity}, we have $\etr{F^TR} = \etr{S\exp(\hat{\eta})}$.
    Let $\Sigma' = (\tr{S}I_{3 \times 3}-S)^{-1} = \diag{\frac{1}{s_2+s_3},\frac{1}{s_1+s_3},\frac{1}{s_1+s_3}} \in\real{3\times 3}$, and let $\xi\in\real{3}$ be defined as $\sqrt{\Sigma'}\xi = \eta$.
    Then, 
	\begin{align*}
		&\etr{F^TR} = \etr{S\exp\left( (\sqrt{\Sigma'}\xi)^\wedge \right)} \\
		=\; &\etr{S \left( I + (\sqrt{\Sigma'}\xi)^\wedge + \frac{1}{2} \left((\sqrt{\Sigma'}\xi)^\wedge\right)^2 + o(\Sigma') \right)} \\
		\propto\; &\etr{\frac{1}{2}S\left((\sqrt{\Sigma'}\xi)^\wedge\right)^2 + o(\Sigma')} \\
		\approx\; &\etr{-\frac{1}{2}\eta^T(\Sigma')^{-1}\eta}.
	\end{align*}
	Also, for the $\nu_R$ term in the conditional mean \eqref{eqn:MFGMiuc}, we have
	\begin{align*}
        \hat\nu_R  & =  \exp\left((\sqrt{\Sigma'}\xi)^\wedge\right)S - S\exp\left((\sqrt{\Sigma'}\xi)^\wedge\right)^T  \\
                   & = \left( I + (\sqrt{\Sigma'}\xi)^\wedge \right)S - S\left( I + (\sqrt{\Sigma'}\xi)^\wedge \right)^T + o\left(\sqrt{\Sigma'}\right) \\
                   & \approx ((\Sigma')^{-1}\eta)^\wedge,
	\end{align*}
	which yields $\mu_c \approx \mu+P(\Sigma')^{-1}\eta$.
	Furthermore  $\Sigma_c = \Sigma - P(\Sigma')^{-1}P^T$. 
    Therefore, \eqref{eqn:MFGDensity} is approximated by a $3+n$ dimensional Gaussian density written in the conditional-marginal form, which is identical to \eqref{eqn:MFGapprox}.
\end{proof}

In other words, when the attitude uncertainty is highly concentrated, an MFG is approximated by a Gaussian distribution on $\real{n+3}$. 
This is because the densities faraway from the mean attitude are negligible, and the densities close to the mean attitude can be approximated by a Gaussian distribution in its tangent space.

\subsection{Moments}

Next, we present selected moments of MFG, which are used in the approximate MLE in the next subsection.
\begin{theorem} \label{thm:MFGMoments}
    Suppose $(R,x)\sim\mathcal{MG}({\mu},{\Sigma},{P},U,S,V)$.
    Then,
    \begin{equation} \label{eqn:MFGER}
        \mathrm{E}[R] = UDV^T,
    \end{equation}
    where $ D =\mathrm{diag}(d_1,d_2,d_3)$ is given in \eqref{eqn:S2D}.
    Also,
    \begin{align}
        &\mathrm{E}[x] = {\mu}, \label{eqn:MFGEx} \\
        &\mathrm{E}[\nu_R] = 0, \label{eqn:MFGEfR} \\
        &\mathrm{E}[xx^T] = {\Sigma}_c+{\mu}{\mu}^T+{P}\mathrm{E}[\nu_R\nu_R^T]{P}^T, \label{eqn:MFGExx} \\
        &\mathrm{E}[x\nu_R^T] = {P}\mathrm{E}[\nu_R\nu_R^T], \label{eqn:MFGExfR}
    \end{align}
    where $\mathrm{E}[\nu_R\nu_R^T]\in\real{3\times3}$ is a diagonal matrix with the $i$-th diagonal element given by
    \begin{equation} \label{eqn:MFGEfRfR}
        (\mathrm{E}[\nu_R\nu_R^T])_{ii} = (s_j^2+s_k^2)\mathrm{E}[Q_{jk}^2]-2s_js_k\mathrm{E}[Q_{jk}Q_{kj}],
    \end{equation}
    for $\{i,j,k\}=\{1,2,3\}$,
	where explicit expressions to calculate $\mathrm{E}[Q_{jk}^2]$ and $\mathrm{E}[Q_{jk}Q_{kj}]$ can be found in \cite{lee2017bayesian}.
\end{theorem}
\begin{proof}
	Equation \eqref{eqn:MFGER} follows immediately from the fact that the marginal distribution of $R$ is a matrix Fisher distribution with parameter $F$.
	Next, for \eqref{eqn:MFGEfR}, we have:
	\begin{align}
		\mathrm{E}[\nu_R] = \mathrm{E}[QS-SQ^T]^\vee = (DS-SD^T)^\vee = 0.
	\end{align}
	Also, for \eqref{eqn:MFGExx}, we can integrate $xx^T$ directly and get
    \begin{align}\label{eqn:ExxTranspose}
		&\mathrm{E}[xx^T] = \int_{\mathrm{SO}(3)}\int_{\mathbb{R}^n}xx^T p(R,x)\mathrm{d}x\mathrm{d}R \nonumber \\
		&= \frac{1}{c(F)}\int_{\mathrm{SO}(3)}[{\Sigma}_c+({\mu}+{P}\nu_R)({\mu}+{P}\nu_R)^T]\mathrm{etr}(FR^T)\mathrm{d}R \nonumber \\
		&= {\Sigma}_c+{\mu}{\mu}^T+{P}\mathrm{E}[\nu_R\nu_R^T]{P}^T.
	\end{align}
	The remaining \eqref{eqn:MFGEx}, \eqref{eqn:MFGExfR} and \eqref{eqn:MFGEfRfR} can be derived similarly by direct integration. 
\end{proof}

\subsection{Maximum Likelihood Estimator} \label{sec:MLE-MFG}

Here we consider the maximum likelihood estimation (MLE) problem to construct an MFG from its samples.
Given a set of samples $(R_i,x_i)_{i=1}^{N_s}$, the log-likelihood function of the parameters, after omitting some constants, is given by
\begin{align} \label{eqn:MFGLogLike}
l = &-\log(c(S))+\mathrm{tr}(F\bar{\mathrm{E}}[R]^T)-\frac{1}{2}\log(\mathrm{det}(\Sigma_c))\nonumber \\
&-\frac{1}{2}\bar{\mathrm{E}}[(x-\mu-P\nu_R)^T\Sigma_c^{-1}(x-\mu-P\nu_R)],
\end{align}
where $\bar{\mathrm{E}}[\cdot]$ represents the sample mean of a random variable, for example, $\bar{\mathrm{E}}[R] = \frac{1}{N_s}\sum_{i=1}^{N_s} R_i$.

As the log-likelihood function must be maximized jointly for the matrix Fisher part and the Gaussian part, it is challenging to obtain a closed form solution for $(\mu,\Sigma, V,S,U,P)$. 
From the construction of MFG, this is comparable to the MLE of a $(9+n)$-variate Gaussian distribution with prescribed linear constraints on the covariance matrix, which is already known as a challenging problem~\cite{zwiernik2017maximum}.

Instead of jointly maximizing the likelihood, we exploit the fact that the marginal distribution for $R$ is a matrix Fisher distribution, and the conditional distribution for $x|R$ is Gaussian. 
More specifically, the log-likelihood for the marginal distribution corresponds to the first two terms on the right hand side of \eqref{eqn:MFGLogLike}, and the marginal MLE for parameters $U,S,V$ is solved by the MLE of the matrix Fisher distribution.

\begin{theorem}[\cite{khatri1977mises,LeeITAC18}] \label{thm:MFMLE}
	The marginal maximum likelihood estimates for $U, V$ are given by the proper singular value decomposition $\expectbar{R}=UDV^T$, and the MLE for $S$ is given by solving \eqref{eqn:S2D} for $S$ using $D$.
\end{theorem}

After obtaining $U,S,V$, they are used in the conditional log-likelihood for $x|R$ corresponding to the last two terms of the right hand side of \eqref{eqn:MFGLogLike}, and the corresponding conditional MLE is addressed as follows. 

\begin{theorem} \label{thm:MFGMLEAppro}
	Let $U,V\in\SO$ and $S\in\real{3\times 3}$ be the solution of the marginal MLE for $R$.
	Define $Q_i = U^TR_iV$, and $\nu_{R_i} = (Q_iS-SQ_i^T)^\vee$ for $i=1,\ldots,N_s$.
	Also define the following sample covariance matrices:
	\begin{align}
		&\overline{\mathrm{cov}}(x,x) = \bar{\mathrm{E}}(xx^T)-\bar{\mathrm{E}}[x]\bar{\mathrm{E}}[x]^T,\\
		&\overline{\mathrm{cov}}(x,\nu_R) = \bar{\mathrm{E}}[x\nu_R^T]-\bar{\mathrm{E}}[x]\bar{\mathrm{E}}[\nu_R]^T, \\
		&\overline{\mathrm{cov}}(\nu_R,\nu_R) = \bar{\mathrm{E}}[\nu_R\nu_R^T]-\bar{\mathrm{E}}[\nu_R]\bar{\mathrm{E}}[\nu_R]^T.
	\end{align}
	Then the solution of the conditional MLE for $P$, $\mu$, and $\Sigma$, is given by
	\begin{align}
		&P = \overline{\mathrm{cov}}(x,\nu_R)\overline{\mathrm{cov}}(\nu_R,\nu_R)^{-1} \label{eqn:MFGMLEP} \\
		&\mu = \bar{\mathrm{E}}[x]-P\bar{\mathrm{E}}[\nu_R] \label{eqn:MFGMLEMu} \\
		&\Sigma = \overline{\mathrm{cov}}(x,x)-P\overline{\mathrm{cov}}(x,\nu_R)^T+P(\tr{S}I_{3\times 3}-S)P^T \label{eqn:MFGMLESigma}.
	\end{align}
\end{theorem}
\begin{proof}
	Take the derivatives of \eqref{eqn:MFGLogLike} with respect to ${\mu}$ and ${P}$ to obtain
	\begin{align*}
		&\frac{\partial l}{\partial\mu} = \Sigma_c^{-1}\bar{\mathrm{E}}\left[x-\mu-{P}\nu_{R}\right], \\
		&\frac{\partial l}{\partial P} = \Sigma_c^{-1}\bar{\mathrm{E}}\left[(x-{\mu})\nu_R^T-{P}\nu_R\nu_R^T\right].
	\end{align*}
	By setting the derivatives zero, the MLE of ${\mu}$ and ${P}$ can be obtained as in \eqref{eqn:MFGMLEMu} and \eqref{eqn:MFGMLEP}.
	Next, take the derivative of \eqref{eqn:MFGLogLike} with respect to $\Sigma_c^{-1}$ to have
	\begin{align*}
		\frac{\partial l}{\partial\Sigma_c^{-1}} = \frac{1}{2}{\Sigma}_c -\frac{1}{2}\bar{\mathrm{E}}\left[(x-\mu-P\nu_{R})
		(x-\mu-{P}\nu_{R})^T\right]. \nonumber
	\end{align*}
    Setting the derivative to zero and substituting \eqref{eqn:MFGMLEP} and \eqref{eqn:MFGMLEMu}, we obtain \eqref{eqn:MFGMLESigma}.
\end{proof}

The given marginal-conditional MLE is only an approximation to the joint MLE, because the information of $U$, $S$ and $V$ encoded in $x_i$ is  discarded over marginalization.
Intuitively, the correlation between $x$ and $\mathrm{vec}(R^T)$ indicated by the samples is not necessarily constrained in the tangent space at $UV^T$ calculated from $\bar{\mathrm{E}}[R]$ as required by the MFG in Theorem \ref{thm:conditioning}.

To understand how well the marginal-conditional MLE approximates the joint MLE, we perform the following analysis to compare the information that $R$ carries about the unknown parameter $S$, with that of $x|R$.
We focus on the specific case when the dimension of the linear part is one, i.e., $n=1$.
Define a metric  $\lambda_{s_i}\in\real{1}$ as
\begin{equation}
	\lambda_{s_i} = \frac{g_{s_is_i}(R)}{g_{s_is_i}(x|R)},
\end{equation}
where $g_{s_is_i}(R)$ is the diagonal element of the Fisher information matrix for the marginal distribution $p(R)$ with respect to $s_i$.
Similarly, $g_{s_is_i}(x|R)$ is for the conditional distribution $p(x|R)$~\cite[Chapter 9.8]{kullback1997information}.
The quantity $\lambda_{s_i}$ indicates the ratio of amount of information of the concentration parameter $s_i$ contained in $R$, versus in $x$, due to the fact that $g_{s_is_i}(R,x) = g_{s_is_i}(R)+g_{s_is_i}(x|R)$.
The higher $\lambda_{s_i}$ is, the less information is discarded in the marginal MLE for parameter $s_i$.

\begin{proposition} \label{prop:InformationRatio}
    Suppose $(x,R)\in\real{1}\times\SO$ follows an MFG with parameters $(\mu,\sigma^2,P,U,sI,V)$ of appropriate dimensions, where $P = \frac{\rho\sigma}{\sqrt{2s}}\begin{bmatrix} 1 & 1 & 1 \end{bmatrix}$ for $\rho\in\real{1}$. Then,
	\begin{equation} \label{eqn:InformationRatio}
		\lambda_{s_i} = \frac{1-3\rho^2}{\rho^2} \frac{2s^2\left(\partial_{11}c(S) - (\partial_1c(S))^2/c(S)\right)}{\partial_1c(S) + s\partial_{11}c(S) - s\partial_{12}c(S)},
	\end{equation}
	where $\partial_i c(S) = \frac{\partial c(S)}{\partial s_i} \big|_{S=sI}$ and $\partial_{ij} c(S) = \frac{\partial^2 c(S)}{\partial s_i \partial s_j} \big|_{S = sI}$.
\end{proposition}
\begin{proof}
	By Chapter 2.6 in \cite{kullback1997information}, we have
	\begin{align*}
		&g_{s_is_i}(R,x) = -\expect{\frac{\partial^2 \log p(R,x)}{\partial s_i^2}} \\
		= &\frac{\partial_{ii}c(S)}{c(S)} - \frac{(\partial_i c(s))^2}{c(S)^2} + \expect{\frac{\partial \nu_R^T}{\partial s_i}P^T\Sigma_c^{-1}P\frac{\partial \nu_R}{\partial s_i}},
	\end{align*}
	where the first two terms are the marginal information $g_{s_is_i}(R)$, and the last expectation is the conditional information $g_{s_is_i}(x|R)$.
	Substitute $n=1$, $\Sigma = \sigma^2$, $S = sI$, and $P = \frac{\rho\sigma}{\sqrt{2s}}\begin{bmatrix} 1 & 1 & 1 \end{bmatrix}$ into the conditional information,
	\begin{align*}
		g_{s_is_i}(x|R) &= \frac{\rho^2}{1-3\rho^2} \frac{\expect{Q_{ji}^2}+\expect{Q_{ki}^2}}{2s} \\
		&= \frac{\rho^2}{1-3\rho^2} \frac{\partial_1c(S) + s\partial_{11}c(S) - s\partial_{12}c(S)}{2s^2c(S)}.
	\end{align*}
	And \eqref{eqn:InformationRatio} follows from the above two equations.
\end{proof}

In Proposition \ref{prop:InformationRatio}, $\rho$ can be interpreted as the \textit{correlation coefficient} between $R$ and $x$.
It is clearly seen from \eqref{eqn:InformationRatio} when $\rho$ is close to zero, i.e., when the correlation between $R$ and $x$ is weak, the information of the matrix Fisher part $s_i$ is mainly contained in $R$.
Therefore, the marginal-conditional MLE is close to the joint MLE. 
Next, we examine the effect of concentration level of the attitude. 
Let $r(s)$ be the second fraction term on the right hand side of \eqref{eqn:InformationRatio}, whose value is illustrated in Fig. \ref{fig:InformationRatio} for varying $s$.
This indicates when the marginal attitude distribution is close to uniform, i.e., $s\rightarrow 0$, relatively more information of $s_i$ is carried by $x$.
On the other hands, when the attitude is more concentrated, say $s > 6$, the fraction $r(s)$ does not vary much, and $\lambda_{s_i}$ is mainly determined by the level of correlation at the first part of the right hand side of \eqref{eqn:InformationRatio}.

The proposed closed form solution for MLE is essential for designing and implementing an attitude estimator based on MFG, as it is inevitably used in each step of uncertainty propagation and measurement update, which typically runs at more than $100\,\mathrm{Hz}$.
It is not feasible to solve the joint MLE numerically at every step for real-time implementations.

\begin{figure}
	\includegraphics{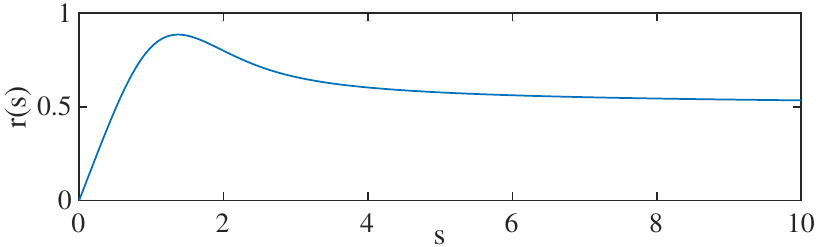}
	\caption{The graph of $r(s) = \frac{2s^2\left(\partial_{11}c(S) - (\partial_1c(S))^2/c(S)\right)}{\partial_1c(S) + s\partial_{11}c(S) - s\partial_{12}c(S)}$ against $s$. \label{fig:InformationRatio}}
\end{figure}

\section{Bayesian Estimation for Attitude and Gyroscope Bias}

In this section, we apply the proposed MFG in the attitude estimation with a time-varying gyro bias.
We design a Bayesian estimator composed of prediction and correction. 
The  prediction step is to propagate the MFG describing the uncertainty of attitude and gyroscope bias along the stochastic differential equation describing the attitude kinematics. 
This is accomplished in two different methods: (i) an analytical scheme by matching a new MFG to the propagated density; and (ii) a sampling-based scheme by using the unscented transform.
In the correction step, we consider two types of measurements, namely attitude and vector measurements.
It is shown that both types lead to the same form of posterior density, which is then matched to a new MFG.

In propagating the uncertainty of attitude and bias, the following kinematics model \cite{LeeITAC18,markley2006attitude} is considered
\begin{gather}
	R^T\diff{R} = (\hat{x}+\hat{\Omega})\diff{t} + (H_udW_u)^\wedge, \label{eqn:SDEAttitude} \\
	\diff{x} = H_vdW_v \label{eqn:SDEBias},
\end{gather}
where $R$ and $x$ are the attitude and bias of the gyroscope, $\Omega\in\real{3}$ is the  angular velocity measured by a gyroscope that is resolved in the body-fixed frame. 
Next, $W_u$ and $W_v$ are two independent 3-dimensional Wiener processes, and $H_u, H_v \in \real{3 \times 3}$ are two matrices describing the strengths of noises.
The angular velocity measurement has two sources of noise: the bias term $x$ and the Gaussian white noise contributed by $H_u dW_u$.
The bias is slowly varying while driven by another white noise $H_v dW_v$. 

The stochastic differential equation \eqref{eqn:SDEAttitude} is interpreted in the Stratonovich sense to guarantee the process does not leave $\SO$ \cite{markley2006attitude,barrau2018stochastic}.
Let the time be discretized by a sequence $\{t_0,t_1,\ldots\}$. 
For convenience, it is assumed that the time step $h\in\real{1}$ is fixed, i.e., $h= t_{k+1} - t_k$. 
According to \cite{barrau2018stochastic}, the kinematics model can be discretized into
\begin{gather}
    R_{k+1} = R_k \exp\left\{ h (\hat{\Omega}_k+\hat{x}_k) + (H_u\Delta W_u)^\wedge \right\}, \label{eqn:DistSDEAttitude} \\
    x_{k+1} = x_k + H_v\Delta W_v,  \label{eqn:DistSDEBias}
\end{gather}
where $\Delta W_u, \Delta W_v$ are the stochastic increments of the Wiener processes over a time step, which are Gaussian with
\begin{equation}\label{eqn:HDeltaW_Gaussian}
	H_u\Delta W_u \sim \mathcal{N}(0,hG_u), \quad H_v\Delta W_v \sim \mathcal{N}(0,hG_v),
\end{equation}
where $G_u = H_uH_u^T$ and $G_v = H_vH_v^T\in\real{3\times 3}$.

The initial attitude and bias $(R(t_0),x(t_0))$ at $t_0$ are assumed to follow an MFG of $n=3$, with the parameters $(\mu_0, \Sigma_0, P_0, U_0, S_0, V_0)$.
We wish to propagate it through the discretized equations \eqref{eqn:DistSDEAttitude} and \eqref{eqn:DistSDEBias} to construct the propagated MFG. 
The evolution of a probability density over a stochastic differential equation is governed by the Fokker-Planck equation, and in general, the propagated density is not necessarily MFG.
This is addressed by calculating the moments of the propagated density, and constructing the corresponding MFG from the solution of MLE presented in Section \ref{sec:MLE-MFG}.
Depending on how the moments of the propagated density are constructed, we present two methods: an analytical method and an unscented method. 

\subsection{Analytical Uncertainty Propagation} \label{sec:analyticalProp}

Suppose $(R_k,x_k)\sim\mathcal{MG}(\mu_k, \Sigma_k, P_k, U_k, S_k, V_k)$. 
In this subsection, we present an approach to construct a new MFG corresponding to the propagated density of $(R_{k+1},x_{k+1})$ by calculating its moments analytically.

First, the exponent in \eqref{eqn:DistSDEAttitude} can be decomposed into
\begin{align*}
    \{h(\Omega_k + \mu_k)\} + \{h(x_k-\mu_k) + H_u \Delta W_u\},
\end{align*}
after taking the hat map off. 
The first part is deterministic, and the second part is a random variable with the zero mean.  
This leads to the following approximation to \eqref{eqn:DistSDEAttitude}.

\begin{lemma}
	Equation \eqref{eqn:DistSDEAttitude} is almost surely equivalent to
	\begin{align} \label{eqn:R_{k+1}factorization}
		R_{k+1} = & R_ke^{h(\hat{x}_k-\hat{\mu}_k) + (H_u\Delta W_u)^\wedge + o(h)} e^{h(\hat{\Omega}_k+\hat{\mu}_k)}.
	\end{align}
\end{lemma}
\begin{proof}
    Equation \eqref{eqn:DistSDEAttitude} is rewritten into 
	\begin{align*} 
		R_{k+1} &= R_k \left[ e^{h(\hat{\Omega}_k+\hat{x}_k) + (H_u\Delta W_u)^\wedge} 
		 e^{-h(\hat{\Omega}_k+\hat{\mu}_k)} \right] e^{h(\hat{\Omega}_k+\hat{\mu}_k)}.
	\end{align*}
    The Baker-Campbell-Hausdorff (BCH) formula \cite{hall2015lie} provides the solution for $Z$ to the equation $e^X e^Y = e^Z$ for given $X,Y$. 
   Applying this to the expression in the square brackets,  
	\begin{align} \label{eqn:R_{k+1}factorizationProof}
		R_{k+1} &= R_k e^{h(\hat{x}_k-\hat{\mu}_k) + (H_u\Delta W_u)^\wedge + A} e^{h(\hat{\Omega}_k+\hat{\mu}_k)},
	\end{align}
    where the additional term $A\in\so$ is composed of terms of order at least $h^2$ and $h\Delta W_u$.
	Since $\lim\limits_{h \to 0}\Delta W_u=0$ almost surely, $A \sim o(h)$.
\end{proof}

This lemma is helpful in making use of the closed form expression for the exponential map $\exp:\so\to\SO$, (i.e., the Rodrigues rotation formula) for the deterministic parts of \eqref{eqn:DistSDEAttitude}.
The uncertainty of $R_{k+1}$ contributed by the noises is quantified by the centered stochastic part $h(\hat{x}_k-\hat{\mu}_k) + (H_u\Delta W_u)^\wedge$ with the zero mean.
Next, we present an expression for $\expect{R_{k+1}}$ to solve the marginal MLE. 

\begin{theorem} \label{thm:ER_{k+1}}
    \begin{align} \label{eqn:E(R_{k+1})}
        \expect{R_{k+1}} & = \left\{ \expect{R_k}\left( I + \frac{h}{2}(G_u-\tr{G_u}I) \right) \right. \nonumber \\
         &\quad \left. + hU_k\expect{Q_kV_k^T\widehat{P_k\nu_{R_k}}} \right\} e^{h(\hat{\Omega}_k+\hat{\mu}_k)} + O(h^2),
    \end{align}
    where $Q_k=U_k^TR_kV_k$, $\nu_{R_k} = (Q_kS_k-S_kQ_k^T)^\vee$.
\end{theorem}
\begin{proof}
	See Appendix \ref{appendix:moments}.
\end{proof}

With the given $\expect{R_{k+1}}$, the marginal MLE of the attitude part of MFG can be solved as discussed in Section \ref{sec:MLE-MFG}, which yields the estimates of $U_{k+1}$, $S_{k+1}$ and $V_{k+1}$.
We define $Q_{k+1} = U_{k+1}^TR_{k+1}V_{k+1}$ and $\nu_{R_{k+1}} = (Q_{k+1}S_{k+1}-S_{k+1}Q_{k+1}^T)^\vee$ as the intermediate parameters for the MFG at time $t_{k+1}$.
Then the conditional MLE for the rest of parameters are solved as in Theorem \ref{thm:MFGMLEAppro} with the moments given as follows.

\begin{theorem} \label{thm:otherMoments}
    Let $\tilde U,\tilde V, \tilde{\tilde{V}} \in\SO$ and $\tilde S, \tilde{\tilde S} \in\real{3\times 3}$ be
    \begin{gather*}
        \tilde{U} = U_{k+1}^T U_k, \quad 
        \tilde{V} = V_{k+1}^T e^{-h(\hat{\Omega}_k+\hat{\mu}_k)} V_k, \quad
        \tilde{S} = \tilde{U}^T S_{k+1} \tilde{V},\\
        \tilde{\tilde{V}} = V_{k+1}^T e^{-h(\hat{\Omega}_k+\hat{\mu})}G_u^TV_k, \quad
        \tilde{\tilde{S}} = \tilde{U}^T S_k \tilde{\tilde{V}}.
    \end{gather*}
    Also, let $\tilde{\nu}_R, \tilde{\tilde\nu}_R \in\real{3}$ and $\Gamma_Q \in\real{3\times 3}$ be
    \begin{gather*}
        \tilde{\nu}_R = (Q_k \tilde{S}^T - \tilde{S} Q_k^T)^\vee, \quad
        \tilde{\tilde{\nu}}_R = (Q_k \tilde{\tilde{S}}^T-\tilde{\tilde{S}} Q_k^T)^\vee,\\
        \Gamma_Q = \left( \tr{Q_k\tilde{S}^T}I - Q_k\tilde{S}^T \right)Q_k.
    \end{gather*}
    Then, the moments of $x_{k+1}$ and $\nu_{R_{k+1}}$ required for the conditional MLE are given by
    \begin{align}
        &\expect{x_{k+1}}  = \mu_k, \label{eqn:Ex_{k+1}} \\
        &\expect{\nu_{R_{k+1}}}  = 0, \label{eqn:EvR_{k+1}} \\
        &\expect{x_{k+1}x_{k+1}^T}  = \expect{x_kx_k^T} + hG_v, \label{eqn:Exx_{k+1}}
    \end{align}
    \begin{align}
        &\expect{x_{k+1}\nu_{R_{k+1}}^T}  = h\Sigma_{c_k}V_k \expect{\Gamma_Q^T}\tilde{U}^T + \mu_k \Big( \expect{\tilde{\nu}^T_R}  \nonumber \\
        &\quad + \frac{1}{2}h\expect{\tilde{\tilde{\nu}}_R^T} - \frac{1}{2}h\tr{G_u}\expect{\tilde{\nu}_R^T} + h\expect{\nu_R^T P_k^T V_k\Gamma_Q^T} \Big)\tilde{U}^T \nonumber \\
        &\quad + P_k \Big( \expect{\nu_{R_k} \tilde{\nu}_R^T} +  \frac{1}{2}h\expect{\nu_R\tilde{\tilde{\nu}}_R^T} - \frac{1}{2}h\tr{G_u} \expect{\nu_{R_k} \tilde{\nu}_R^T} \nonumber \\
        &\quad + h\expect{\nu_{R_k}\nu_{R_k}^T P_k^T V_k \Gamma_Q^T} \Big)\tilde{U}^T + O(h^2), \label{eqn:ExvR_{k+1}}\\
        &\expect{\nu_{R_{k+1}} \nu_{R_{k+1}}^T } = \tilde{U} \Big( \expect{\tilde{\nu}_R\tilde{\nu}^T_R} - h\tr{G_u} \expect{\tilde{\nu}_R\tilde{\nu}_R^T} \nonumber \\
        &\quad + \frac{1}{2}h\expect{\tilde{\nu}_R\tilde{\tilde{\nu}}^T_R} + \frac{1}{2}h\expect{\tilde{\tilde{\nu}}_R\tilde{\nu}^T_R} + h\expect{\tilde{\nu}_R\nu_{R_k}^T P_k^T V_k \Gamma_Q^T} \nonumber \\
        &\quad + h\expect{\tilde{\nu}_R\nu_{R_k}^T P_k^T V_k \Gamma_Q^T}^T + h\expect{\Gamma_Q V_k^T G_u V_k\Gamma_Q^T} \Big)\tilde{U}^T \nonumber \\
        &\quad + O(h^2).\label{eqn:EvRvR_{k+1}}
    \end{align}
\end{theorem}
\begin{proof}
	See Appendix \ref{appendix:moments}.
\end{proof}

\begin{remark}
	Note that $\nu_{R_k}$, $\tilde{\nu}_R$, $\tilde{\tilde{\nu}}_R$ are linear in $Q_k$, and by Lemma \ref{lemma:GammaQSimplified} in Appendix \ref{appendix:moments}, $\Gamma_Q$ is also linear in $Q_k$.
	Therefore, the expectations on the right hand side of equations \eqref{eqn:E(R_{k+1})}, \eqref{eqn:ExvR_{k+1}} and \eqref{eqn:EvRvR_{k+1}} can be calculated using the moments of $Q$ of the matrix Fisher distribution with the parameter $S_k$ up to the third order.
    More specifically, these moments can be expressed as linear combinations of $\expect{Q_{ij}}$, $\expect{Q_{ij}Q_{kl}}$, $\expect{Q_{ij}Q_{kl}Q_{mn}}$ for $i,j,k,l,m,n \in \{1,2,3\}$.
    The moments for $Q_k$ can be calculated using \eqref{eqn:S2D} and the method introduced in \cite{wang2020higher}.
\end{remark}

With these moments, the estimates for $(\mu_{k+1},\Sigma_{k+1},P_{k+1})$ can be constructed through the conditional MLE given in Theorem \ref{thm:MFGMLEAppro}. 

In summary, these provide an analytical approach to propagate $(R_k,x_k)\sim\mathcal{MG}(\mu_k, \Sigma_k, P_k, U_k, S_k, V_k)$ into $(R_{k+1},x_{k+1})\sim\mathcal{MG}(\mu_{k+1}, \Sigma_{k+1}, P_{k+1}, U_{k+1}, S_{k+1}, V_{k+1})$, whose moments match with the solution of \eqref{eqn:DistSDEAttitude} and \eqref{eqn:DistSDEBias} up to $O(h^2)$.
This is a typical procedure in assumed density filters, and its validity is guaranteed by the fact that the maximum likelihood estimator minimizes the Kullback–Leibler divergence from the MFG family to the true density \cite{white1982maximum}.
The pseudocode for this analytical uncertainty propagation scheme is shown in Table \ref{tab:AnalyticalUP}.

\begin{table}
    \caption{Analytical Uncertainty Propagation \label{tab:AnalyticalUP}}
    \begin{algorithmic}[1]
        \algrule[0.8pt]
        \Procedure{$\mathcal{MG}(t_{k+1}) = $ Analytical Propagation}{$\mathcal{MG}(t_k),\Omega_k$}
        \algrule
        \State Calculate $\expect{R_{k+1}}$ using \eqref{eqn:E(R_{k+1})}.
        \State Obtain $U_{k+1},S_{k+1},V_{k+1}$ according to Theorem \ref{thm:MFMLE} using $\expect{R_{k+1}}$.
        \State Calculate the moments in Theorem \ref{thm:otherMoments}.
        \State Obtain $\mu_{k+1},\Sigma_{k+1},P_{k+1}$ according to Theorem \ref{thm:MFGMLEAppro} using the moments calculated in Step 4.
        \State Set $\mathcal{MG}(t_{k+1}) = \mathcal{MG}(\mu_{k+1},\Sigma_{k+1},P_{k+1},U_{k+1},S_{k+1},V_{k+1})$.
        \EndProcedure
        \algrule[0.8pt]
    \end{algorithmic}
\end{table}

\subsection{Unscented Uncertainty Propagation}

In this subsection, we present an alternative sampling-based method to propagate MFG.
In contrast to the preceding analytical approach to propagate moments, the unscented transform selects so called sigma points from the distribution of $(R_k,x_k)$, which are propagated through \eqref{eqn:DistSDEAttitude} and \eqref{eqn:DistSDEBias}, and matched to a new MFG using the approximate MLE.
The sigma points are selected in a deterministic fashion to characterize the mean and dispersion of the distribution.
To formulate the unscented transform, we first introduce a canonical form of MFG by nullifying all the correlations and mean values.

\begin{theorem}
	Let $(R,x) \sim \mathcal{MG}(\mu,\Sigma,P,U,S,V)$. 
    Define $Q=U^TRV\in\SO$, $y=\Sigma_c^{-1/2}(x-\mu-P\nu_R)\in\real{n}$, then $(Q,y)$ follows $\mathcal{MG}(0,I,0,I,S,I)$.
\end{theorem}
\begin{proof}
	Let $\mathrm{d}R$ be the probability measure of the uniform distribution on \SO, i.e. $\int_{\SO}\mathrm{d}R=1$, and let $\mathrm{d}x$ denote the usual Lebesgue measure in $\mathbb{R}^n$.
	After change of variables, the product measure becomes $\mathrm{d}y\mathrm{d}{Q} = \mathrm{det}({\Sigma}_c^{-1/2})\mathrm{d}x\mathrm{d}R$.
    Substituting this and the expressions for $R,x$ in terms of $Q,y$ into \eqref{eqn:MFGDensity}, the probability measure of $({Q},y)$ becomes
	\begin{align*}
		p(R,x)\mathrm{d}x\mathrm{d}R =  \frac{1}{c(S)\sqrt{(2\pi)^n}}\exp(-\frac{yy^T}{2})\mathrm{etr}(S{Q}^T)\mathrm{d}y\mathrm{d}{Q},
	\end{align*}
	which completes the proof.
\end{proof}

In the above canonical form, the Gaussian part is decoupled from the matrix Fisher part. 
Therefore, the sigma points of the canonical MFG can be constructed by the union of the sigma points for the Gaussian distribution and those for the matrix Fisher distribution~\cite{LeeITAC18}.
These can be transformed back with the inverse of the canonical transform to become the sigma points of the original MFG.

\begin{definition} \label{def:MFGSigmaPoints}
	Consider $\mathcal{MG}({\mu},{\Sigma},{P},U,S,V)$.
	Define the $7+2n$ sigma points for its canonical distribution as
	\begin{gather}
		({Q},y)_{1,2} = \left(\exp(\pm\theta_1\hat{e}_1),[0,\ldots,0]^T\right), \nonumber \\
		({Q},y)_{3,4} = \left(\exp(\pm\theta_2\hat{e}_2),[0,\ldots,0]^T\right), \nonumber \\
		({Q},y)_{5,6} = \left(\exp(\pm\theta_3\hat{e}_3),[0,\ldots,0]^T\right), \nonumber \\
		({Q},y)_{7,8} = \left(I_{3\times3},\left[\pm\sqrt{\frac{n}{w_G}},0,\ldots,0\right]^T\right), \nonumber \\
		\vdots \nonumber \\
		({Q},y)_{5+2n,6+2n} = \left(I_{3\times3},\left[0,\ldots,0,\pm\sqrt{\frac{n}{w_G}}\right]^T\right), \nonumber \\
		({Q},y)_{7+2n} = \left(I_{3\times3},[0,\ldots,0]^T\right),
	\end{gather}
	where $w_G>0$ and $0\leq \theta_i\leq \pi$ is chosen according to
	\begin{subnumcases}{\label{eqn:MFUnscentedCos} \cos\theta_i=}
		\sigma + \dfrac{(1-\sigma)(\log c(S) - s_i)}{s_j+s_k},\;  \mbox{if $s_j+s_k\geq 1$},\label{eqn:costhetaia}\\
		\{\sigma + (1-\sigma)(\log c(S) - s_i)+\frac{1}{2}\}(s_j+s_k)\nonumber\\ -1/2, \quad \mbox{else if $0\leq s_j+s_k < 1$,}\label{eqn:costhetaib}
	\end{subnumcases}
	for $\{i,j,k\}=\{1,2,3\}$.
	The weights for the first three pairs of sigma points are given by
	\begin{align}
		w_i = \frac{1}{4(1-\cos\theta_i)}\left\{\frac{1}{c(S)}\left(\frac{\partial c(S)}{\partial s_i}-\frac{\partial c(S)}{\partial s_j}-\frac{\partial c(S)}{\partial s_k}\right)+1\right\}
	\end{align}
	for $i=1,2,3$, where $\sigma$ in \eqref{eqn:MFUnscentedCos} is chosen such that $2(w_1+w_2+w_3)=w_M$ for a $w_M>0$.
	The weights for the next $2n$ sigma points, namely from the $7$-th through $(6+2n)$-th sigma points are $\frac{w_G}{2n}$, and the weight for the last one is $w_0=1-w_M-w_G$.
	Let $R_i = U{Q}_iV^T$, and $x_i={\Sigma}_c^{1/2}y_i+{\mu}+{P}\nu_{R_i}$.
    The sigma points for $\mathcal{MG}({\mu},{\Sigma},{P},U,S,V)$ are defined as $\{(R,x)_i\}_{i=1}^{7+2n}$.
\end{definition}

In other words, each pair of sigma points are designed to capture the dispersion along each principal axis of MFG.
Parameters $w_M$, $w_G$ and $w_0$ are used to adjust the weights, respectively for the attitude, the linear components and the sigma point at identity.
However, because $\SO$ is compact, unlike the sigma points for a Gaussian distribution on $\real{n}$ which can extend arbitrarily faraway from the origin, the angle $\theta_i$ in Definition \ref{def:MFGSigmaPoints} is constrained by $\pi$.
Therefore, the weight at $I_{3\times3}$, i.e. $w_G+w_0$ cannot be too large if the attitude distribution is very dispersed, which makes \eqref{eqn:MFUnscentedCos} unsolvable since $\theta_i$ cannot be far enough to capture the large dispersion.
To deal with this problem, we note that $\cos\theta_i$ increases monotonically with $\sigma$, so we may set a lower bound for $\cos\theta_i$ denoted as $\cos\theta_l$, such that if the $\sigma$ solved from $w_M$ is smaller than that solved from $\min_{i=1,2,3}\{\cos\theta_i\} = \cos\theta_l$, we choose the latter to make $\theta_i\leq\theta_l$.
In implementation, $\cos\theta_l=-\sqrt{3}/2$ shows good numerical stability.

This selection of sigma points in Definition \ref{def:MFGSigmaPoints} is justified by the following theorem stating that the original MFG can be recovered when the MLE described in Section \ref{sec:MLE-MFG} is performed for the sigma points.
 
\begin{theorem}
	The marginal-conditional MLE of the sigma points given in Definition \ref{def:MFGSigmaPoints} are $({\mu},{\Sigma},{P},U,S,V)$.
\end{theorem}
\begin{proof}
	The marginal MLE for $U$, $S$ and $V$ from the sigma points has been shown in Theorem IV.I~\cite{LeeITAC18}.
    Since $\bar{\mathrm{E}}[y]=0$, $\bar{\mathrm{E}}[yy^T]=I_{n\times n}$, $\bar{\mathrm{E}}[\nu_R]=0$, and $\bar{\mathrm{E}}[y\nu_R^T]=0$, we have
	\begin{align}
	&\overline{\mathrm{cov}}(x,\nu_R) = {P}\overline{\mathrm{cov}}(\nu_R,\nu_R) \\
	&\bar{\mathrm{E}}[x] = {\mu} \\
	&\overline{\mathrm{cov}}(x,x) = {\Sigma}_c+{P}\overline{\mathrm{cov}}(\nu_R,\nu_R){P}^T.
	\end{align}
	According to Theorem \ref{thm:MFGMLEAppro}, these show the MLE of $P$, $\mu$ and $\Sigma$ are the same as the original MFG.
\end{proof}

Next, we introduce how to propagate the uncertainty of attitude and gyro bias using the unscented transform.
Given $(R_k,x_k)\sim\mathcal{MG}(\mu_k,\Sigma_k,P_k,U_k,S_k,V_k)$, we select $7+2n= 13$ sigma points, together with 7 sigma points from the noise $H_u\Delta W_u$ in \eqref{eqn:HDeltaW_Gaussian} according to any common unscented transform for a Gaussian distribution in $\real{3}$ (for example, see \cite[Chapter 8-11]{haug2012bayesian}).
These sigma points are propagated to $t_{k+1}$ through the discrete kinematics model \eqref{eqn:DistSDEAttitude} and \eqref{eqn:DistSDEBias} without the noise term $H_v\Delta W_v$.
Then a new MFG at $t_{k+1}$ is recovered from these propagated sigma points using the marginal-conditional MLE.
The effect of the noise term $H_v\Delta W_v$ driving the gyro bias in \eqref{eqn:DistSDEBias} is accounted by adding $hG_v$ to the new covariance matrix $\Sigma_{k+1}$ for $x_{k+1}$, according to Proposition \ref{prop:MFGAddN}.
The pseudocode for this uncertainty propagation scheme is summarized in Table \ref{tab:UnscentedUP}.

\begin{proposition} \label{prop:MFGAddN}
	Suppose $(R,x)\sim\mathcal{MG}(\mu,\Sigma,P,U,S,V)$, and $x'\sim\mathcal{N}(\mu',\Sigma')$ is independent of $(R,x)$. Then $(R,x+x') \sim \mathcal{MG}(\mu+\mu',\Sigma+\Sigma',P,U,S,V)$.
\end{proposition}
\begin{proof}
	Let $y=x+x'$, then the density function for $(R,y)$ is
	\begin{align*}
		&f_{R,y}(R,y) = \int_{x\in\real{n}}f_{R,x}(R,x)f_{x'}(y-x)\diff{x} \\
		&\quad = \frac{1}{c}\int_{x\in\real{n}} \etr{FR^T} \exp\left(-\frac{1}{2}(x-\mu_c)^T\Sigma_c^{-1}(x-\mu_c)\right) \\ 
		&\quad\quad\times \exp\left(-\frac{1}{2}(y-x-\mu')^T(\Sigma')^{-1}(y-x-\mu')\right) \diff{x} \\
		&\quad = \frac{1}{c'}\etr{FR^T} \exp\left\{-\frac{1}{2}(y-\mu_c-\mu')^T(\Sigma_c+\Sigma')^{-1} \right. \\
		&\quad\quad \times (y-\mu_c-\mu') \bigg\},
	\end{align*}
	where $c$, $c'$ are some normalizing constants, and the last equality comes from the standard proof for adding two independent Gaussian random vectors.
	It suffices to compare the above equation with \eqref{eqn:MFGDensity}.
\end{proof}

Compared to the analytical uncertainty propagation method, the unscented transform uses sigma points as samples to approximate the moments necessary for MLE, instead of calculating them directly.

\begin{table}
	\caption{Unscented Uncertainty Propagation \label{tab:UnscentedUP}}
	\begin{algorithmic}[1]
		\algrule[0.8pt]
        \Procedure{$\mathcal{MG}(t_{k+1}) = $ Unscented Propagation}{$\mathcal{MG}(t_k),\Omega_k$}
		\algrule
		\State Select sigma points and weights $\{(R,x,w)_i\}_{i=1}^{13}$ from $\mathcal{MG}(t_k)$.
		\State Select sigma points and weights $\{(H_u\Delta W_u,w)_{j}\}_{j=1}^7$ from $\mathcal{N}({0},hG_u)$ according to any unscented transform of a Gaussian distribution \cite{haug2012bayesian}.
		\State Propagate the sigma points through \eqref{eqn:DistSDEAttitude} and \eqref{eqn:DistSDEBias} without the noise $H_v\Delta W_v$, i.e.
		\begin{align*}
			R_{i,j} &= R_i\exp(h(\hat{\Omega}_k + \hat{x}_i) + (H_u\Delta W_u)_j^\wedge), \qquad x_{i,j} &= x_i,
		\end{align*}
		and calculate the weights as $w_{i,j}=w_iw_j$.
        \State Obtain $(\mu_{k+1},\Sigma_{k+1},P_{k+1},U_{k+1},S_{k+1},V_{k+1})$ from these $13 \times 7 = 91$ sigma points $(R,x)_{i,j}$ using Theorem \ref{thm:MFMLE} and Theorem \ref{thm:MFGMLEAppro}.
        \State Let $\Sigma_{k+1} = \Sigma_{k+1}+hG_v$
        \State Set $\mathcal{MG}(t_{k+1}) = \mathcal{MG}(\mu_{k+1},\Sigma_{k+1},P_{k+1},U_{k+1},S_{k+1},V_{k+1})$.
		\EndProcedure
		\algrule[0.8pt]
	\end{algorithmic}
\end{table}

\subsection{Measurement Update} \label{sec:MeasurementUpdate}

Next, we present how to update the propagated MFG when measurements are available.
Here we consider the cases when the attitude is directly measured or some vectors associated with attitude, such as magnetic or gravity direction, are measured.
As the measurement update is assumed to be completed instantaneously, the subscript $k$ denoting the time step is omitted throughout this subsection. 
The variables relevant to the posterior distribution conditioned by measurements are denoted by the superscript $+$.

First, suppose the attitude is measured by $N_a$ attitude sensors as $Z_i\in\SO$, which are disturbed by matrix Fisher distributed errors with the parameters $F_{Z_i}\in\real{3\times 3}$ for $i=1,\ldots,N_a$.
More specifically, given the true attitude $R_t\in\SO$, the measurement error $R_t^TZ_i\in\SO$ follows the matrix Fisher distribution with the parameter $F_{Z_i}$, which characterizes the accuracy and the bias of the $i$-th attitude sensor. 

Next, suppose there are also $N_v$ reference vectors ${a}_j\in\mathbb{S}^2$ expressed in the inertial reference frame, which are measured by direction sensors in the body-fixed frame as ${z}_j\in\mathbb{S}^2$ for $j=1,\ldots,N_v$.
Furthermore, given the true attitude $R_t$, assume the noisy measurement ${z}_j$ follows the Von Mises Fisher distribution defined on $\mathbb{S}^2$ \cite{mardia2009directional} with mean direction $R_t^TB_j{a}_j\in\mathbb{S}^2$ and concentration parameter $\kappa_j>0$.
The parameter $B_j\in\SO$ specifies the constant bias of the direction sensor, and $\kappa_j$ specifies the concentration of its random noise.

Suppose the prior distribution of $(R,x)$ before measurements is an MFG with parameters $({\mu},{\Sigma},{P},U,S,V)$. 
By Bayes' rule and Theorem III.2 in \cite{LeeITAC18}, the posterior density conditioned on all of the available measurements $\mathcal{Z} = \{Z_1,\ldots Z_{N_a},{z}_1,\ldots,{z}_{N_v}\}$ is 
\begin{align} \label{eqn:PostDensity}
    p(R,x\lvert \mathcal{Z} ) & \propto \mathrm{etr}\left(\left(F+\sum_{i=1}^{N_a}Z_iF_i^T+\sum_{j=1}^{N_v}\kappa_jB_j{a}_j{z}_j^T\right)R^T\right) \nonumber \\
	&\quad \times \mathrm{exp}\left(-\frac{1}{2}(x-{\mu}_c)^T{\Sigma}_c^{-1}(x-{\mu}_c)\right),
\end{align}
where $F$, ${\mu}_c$ and ${\Sigma}_c$ are defined as in Definition \ref{def:MFG} with respect to $({\mu},{\Sigma},{P},U,S,V)$.
The above posterior distribution of $(R,x)$ is no longer an MFG, as the tangent space at the mean attitude of the updated matrix Fisher part is altered, i.e., the correlation does not satisfy the constraint described in Theorem \ref{thm:conditioning}. 
Similar to the previous two subsections, we match an MFG with parameters $(\mu^+,\Sigma^+,P^+,U^+,S^+,V^+)$ to this density through the marginal-conditional MLE after calculating the moments required.

\begin{theorem} \label{thm:PostMoments}
	Define  $F^+\in\real{3\times 3}$ as
	\begin{align} \label{eqn:FPosterior}
		F^+ = F+\sum_{i=1}^{N_a}Z_iF_i^T+\sum_{j=1}^{N_v}\kappa_j B_j{a}_j{z}_j^T,
	\end{align}
	and let its proper singular value decomposition be $F^+ = U^+ S^+ (V^+)^T$.
	Also, let 
	\begin{align}
		\nu^+_R = (Q^+S^+-S^+(Q^+)^T)^\vee,
	\end{align}
	for $Q^+ = (U^+)^T R V^+\in\SO$.
    Then the moments of the posterior density \eqref{eqn:PostDensity}, namely $\expect{ R\lvert \mathcal{Z} }$, $\expect{ \nu^+_R\lvert \mathcal{Z} }$ and $\expect{ \nu^+_R(\nu^+_R)^T\lvert \mathcal{Z} }$ are identical to their counterparts in Theorem \ref{thm:MFGMoments} after replacing $U,S,V$ with $U^+,S^+,V^+$, and
	\begin{align}
        \expect{ x \lvert \mathcal{Z} } & = {\mu}+{P}\expect{ \nu_R \lvert \mathcal{Z} }, \label{eqn:PostEx} \\
        \expect{ xx^T \lvert \mathcal{Z} } & = {\mu}{\mu}^T+{\mu}\expect{ \nu_R \lvert \mathcal{Z} }^T{P}^T+{P}\expect{ \nu_R \lvert \mathcal{Z} }{\mu}^T \nonumber \\
                                            &\quad+{P}\expect{ \nu_R\nu_R^T \lvert \mathcal{Z} }{P}^T+{\Sigma}_c,\\
        \expect{ x(\nu^+_R)^T \lvert \mathcal{Z} } & = {P}\expect{ \nu_R(\nu^+_R)^T \lvert \mathcal{Z} } \label{eqn:PostExv'R},
	\end{align}
	where
	\begin{align}
        \expect{ \nu_R \lvert \mathcal{Z} } & =\tilde{U}(\expect{Q^+ \lvert \mathcal{Z}} \tilde{S}^T-\tilde{S}\expect{Q^+ \lvert \mathcal{Z}}^T)^\vee, \label{eqn:PostvR} \\
        \expect{ \nu_R\nu_R^T \lvert \mathcal{Z} } & = \tilde U\expect{ \tilde{\nu}^+_R(\tilde{\nu}^+_R)^T \lvert \mathcal{Z} }\tilde U^T, \\
        \expect{ \nu_R(\nu^+_R)^T \lvert \mathcal{Z} } & = \tilde U\expect{ \tilde{\nu}^+_R(\nu^+_R)^T \lvert \mathcal{Z} }, \label{eqn:PostvRv'R}
	\end{align}
    with $\tilde U = U^TU^+$, $\tilde V = V^TV^+\in\SO$, $\tilde{S} = \tilde U^T{S}\tilde V\in\real{3\times 3}$ and $\tilde\nu^+_R\in\real{3}$ is
	\begin{align*}
		\tilde{\nu}^+_R = (Q^+\tilde{S}^T-\tilde{S}(Q^+)^T)^\vee.
	\end{align*}
\end{theorem}
\begin{proof}
    The expressions for $\expect{R \lvert \mathcal{Z}}$, $\expect{\nu^+_R \lvert \mathcal{Z}}$, $\expect{\nu^+_R(\nu^+_R)^T \lvert \mathcal{Z}}$, and \eqref{eqn:PostEx}-\eqref{eqn:PostExv'R} can be obtained by integrating these variables with respect to the density \eqref{eqn:PostDensity}.
	Since $ Q = U^T R V = \tilde U Q^+ \tilde V^T$, we have
	\begin{equation*}
		\nu_R = (\tilde UQ^+\tilde V^TS-S\tilde VQ^+ \tilde U^T)^\vee = \tilde{U}(Q^+\tilde{S}^T-\tilde{S}(Q^+)^T)^\vee,
	\end{equation*}
	from which \eqref{eqn:PostvR}-\eqref{eqn:PostvRv'R} follow.
\end{proof}

\begin{remark}
	$\expect{\tilde{\nu}_R^+ \tilde{\nu}_R^+ \lvert \mathcal{Z}}$ and $\expect{\tilde{\nu}_R^+\nu_R^+ \lvert \mathcal{Z}}$ can be expressed as linear combinations of $\expect{Q^+_{ij}Q^+_{kl} \lvert \mathcal{Z}}$ for $i,j,k,l\in\{1,2,3\}$, therefore they can be calculated using the second order moments of the matrix Fisher distribution with parameter $S^+$.
\end{remark}

Since the attitude part of \eqref{eqn:PostDensity} is already a matrix Fisher density, $U^+S^+(V^+)^T = F^+$ is the solution to the marginal MLE for the matrix Fisher part.
The conditional MLE is solved by Theorem \ref{thm:MFGMLEAppro} with the moments calculated above, which yields $\mu^+$, $\Sigma^+$ and $P^+$.
This provides the measurement update to represent the posterior distribution conditioned by the measurement as an MFG.

The proposed propagation and correction steps constitute a Bayesian attitude and gyro bias estimator. 
The current belief represented by an MFG can be propagated until an additional measurement is available, from which the propagated belief is updated. 
The estimates for the attitude and gyro bias are given by $UV^T$ and $\mu$, respectively. 
The pseudocode for the proposed Bayesian estimation scheme is presented in Table \ref{tab:BayesianFilter}.
The MATLAB code can be accessed at \cite{MFGCode}.

\begin{table}
	\caption{Bayesian Filter for Attitude and Gyroscope Bias \label{tab:BayesianFilter}}
	\begin{algorithmic}[1]
		\algrule[0.8pt]
		\Procedure{Estimation}{$\mathcal{MG}(t_0),\Omega(t)$}
		\algrule
		\State Let $k=0$.
		\Repeat
        \State Either $\mathcal{MG}(t_{k+1})$ = Analytical Propagation($\mathcal{MG}(t_k)$,$\Omega(t_k)$) or $\mathcal{MG}(t_{k+1})$ = Unscented Propagation($\mathcal{MG}(t_k)$,$\Omega(t_k)$).
		\State $k=k+1$.
		\Until $Z(t_{k+1})$ or $z(t_{k+1})$ ia available
		\State $\mathcal{MG}(t_{k+1})$ = Measurement Update($\mathcal{MG}(t_{k+1})$,$Z(t_{k+1})$,$z(t_{k+1})$).
		\State Obtain the estimates as $R(t_{k+1})=UV^T$, $x(t_{k+1})=\mu$ for $\mathcal{MG}(t_{k+1})$.
		\State \textbf{go to} step 3.
		\EndProcedure
		\algrule
		\Procedure{$\mathcal{MG}^+$ = Measurement Update}{$\mathcal{MG}^-$,$Z$,$z$}
        \State Compute $F^+$ from \eqref{eqn:FPosterior}, and calculate its proper SVD as $U^+S^+(V^+)^T=F^+$. 
		\State Calculate the moments of the posterior density in Theorem \ref{thm:PostMoments}.
		\State Obtain $\mu^+,\Sigma^+,P^+$ according to Theorem \ref{thm:MFGMLEAppro}.
		\State Set $\mathcal{MG}^+ = \mathcal{MG}(\mu^+,\Sigma^+,P^+,U^+,S^+,V^+)$
		\EndProcedure
		\algrule[0.8pt]
	\end{algorithmic}
\end{table}

\section{Numerical Simulations}

In this section, we compare the proposed Bayesian filter based on MFG with the well-established MEKF through simulation studies.
We consider a rotational motion of a rigid body where the three Euler angles (body-fixed 3-2-1) follow sinusoidal waves with the frequency at \SI{0.35}{\hertz}, and the amplitudes of $\pi$, $\pi/2$ and $\pi$, respectively.
The corresponding average angular speed is \SI{6.17}{\radian/\second}.
The measured angular velocity is obtained from its true value by adding a bias and a white noise.
The gyroscope bias is modeled as a Wiener process (bias-instability noise) starting at zero with the isotropic strength $\sigma_v = $ \SI{500}{deg/\hour/\sqrt{\second}}, i.e., $H_v = \sigma_vI_{3\times3}$ in \eqref{eqn:SDEBias}.
The white noise of angular velocity is Gaussian (angle random walk) with the isotropic strength $\sigma_u = $ \SI{10}{deg/\sqrt{\second}}, i.e., $H_u = \sigma_uI_{3\times3}$ in \eqref{eqn:SDEAttitude}.
These two values are greater than those of typical gyroscopes, but they are selected to generate large uncertainties favored by the MFG distribution.

The attitude measurement model for the proposed Bayesian estimator with MFG is given by
\begin{equation} \label{eqn:MeaMFG}
	Z = R_t \delta R,
\end{equation}
where $R_t\in\SO$ is the true attitude, and $\delta R\in\SO$ follows a matrix Fisher distribution with a diagonal matrix parameter $S_m\in\real{3\times 3}$.
This corresponds to the attitude measurement model introduced in Section \ref{sec:MeasurementUpdate} without any constant bias.
The measurement equation for MEKF is not written as \eqref{eqn:MeaMFG}. 
Instead, it is given by
\begin{equation} \label{eqn:MeaMEKF}
	Z = R_t \exp(\delta\hat{\theta}),
\end{equation}
where $\delta\theta\in\real{3}$ and $\delta\theta\sim\mathcal{N}(0,\Sigma_m)$.

In order to fairly compare the proposed estimator based on MFG with MEKF, each estimator is simulated with both of \eqref{eqn:MeaMFG} and \eqref{eqn:MeaMEKF} after conversion. 
For example, when \eqref{eqn:MeaMFG} is used in MEKF, $\log(\delta R)^\vee$, which corresponds to $\delta\theta$ in \eqref{eqn:MeaMEKF}, is fitted into a Gaussian distribution.
Conversely, if \eqref{eqn:MeaMEKF} is used in the MFG filter, $\exp(\delta\hat{\theta})$, which corresponds to  $\delta R$ in \eqref{eqn:MeaMFG}, is fitted to a matrix Fisher distribution. 

The sampling frequencies for the gyroscope and the attitude measurements are  \SI{150}{\hertz}, and \SI{30}{\hertz}, respectively. 
Next, we consider three cases depending on the level of the noise and the initial error.
To test the estimators in a statistical sense, under each condition, sixty Monte Carlo simulations were conducted with the simulation time of 60 seconds.
The attitude and bias errors are averaged in each simulation, and further they are averaged over sixty simulations.
Paired $t$-tests ($N=60,\alpha=0.05$) were conducted between the two types of the proposed estimators based on MFG and MEKF.
Any significant statistical difference is indicated by $p<0.05$.

\subsection{Small Initial Errors} \label{sec:simA}

In this subsection, the measurement noises are chosen as isotropic with $S_m = s_mI_{3\times 3}$ and $\Sigma_m = \sigma_m^2I_{3\times 3}$ for $s_m,\sigma_m\in\real{1}$.
Three different levels of concentration are tested: (i) $s_m=2.4$, $\sigma_m=0.5$; (ii) $s_m=12$, $\sigma_m=0.2$; (iii) $s_m=200$, $\sigma_m=0.05$.
Therefore, case (i) has the greatest measurement noise among the three cases considered, and case (iii) has the lowest measurement noise. 
The initial distribution for attitude is the same as $Z$ in \eqref{eqn:MeaMFG} (or \eqref{eqn:MeaMEKF} for MEKF), and the initial distribution for bias is $x_0\sim \mathcal{N}(0,0.1^2I_{3\times3})$, and they are independent.
The initial attitude and bias estimates are drawn randomly from these distributions.

\begin{table*}
	\centering
	\caption{Attitude (\SI{}{deg}) and bias errors (\SI{}{deg/\second}) ($\pm$ sd): small initial errors \label{tab:MFGvsMEKFA}}
	\begin{tabular}{l|c|ccc}
        Measurement Model &  & MFG Analytical ($p$) & MFG Unscented ($p$) & MEKF \\ \hline \hline
        \multirow{2}{*}{\eqref{eqn:MeaMEKF} with $\sigma_m$=0.5} & attitude & 11.94$\pm$0.40 ($<$0.001) & 11.94$\pm$0.40 ($<$0.001) & 11.88$\pm$0.41 \\
		& bias & 4.6$\pm$1.4 (0.81) & 4.6$\pm$1.4 (0.69) & 4.6$\pm$1.5 \\ \hline
        \multirow{2}{*}{\eqref{eqn:MeaMFG} with $s_m$=2.4} & attitude & 12.00$\pm$0.49 ($<$0.001) & 12.00$\pm$0.49 ($<$0.001) & 12.13$\pm$0.48 \\
		& bias & 5.0$\pm$1.7 (0.46) & 4.9$\pm$1.7 (0.42) & 5.0$\pm$1.8 \\ \hline \hline
		\multirow{2}{*}{\eqref{eqn:MeaMEKF} with $\sigma_m$=0.2} & attitude & 7.285$\pm$0.151 ($<$0.001) & 7.285$\pm$0.151 ($<$0.001) & 7.284$\pm$0.152 \\
		& bias & 3.79$\pm$1.11 (0.41) & 3.79$\pm$1.12 (0.47) & 3.79$\pm$1.13 \\ \hline
		\multirow{2}{*}{\eqref{eqn:MeaMFG} with $s_m$=12} & attitude & 7.449$\pm$0.182 (0.009) & 7.449$\pm$0.182 (0.015) & 7.450$\pm$0.181 \\
		& bias & 3.89$\pm$1.14 (0.09) & 3.89$\pm$1.14 (0.033) & 3.90$\pm$1.15 \\ \hline \hline
		\multirow{2}{*}{\eqref{eqn:MeaMEKF} with $\sigma_m$=0.05} & attitude & 3.54737$\pm$0.04250 (0.64) & 3.54741$\pm$0.04249 (0.009) & 3.54736$\pm$0.04251 \\
		& bias & 3.356$\pm$0.900 (0.97) & 3.357$\pm$0.903 (0.88) & 3.356$\pm$0.910 \\ \hline
		\multirow{2}{*}{\eqref{eqn:MeaMFG} with $s_m$=200} & attitude & 3.55406$\pm$0.04607 (0.21) & 3.55414$\pm$0.04606 (0.035) & 3.55409$\pm$0.04605 \\
		& bias & 3.722$\pm$0.939 (0.053) & 3.726$\pm$0.940 (0.031) & 3.733$\pm$0.940 \\ \hline \hline
	\end{tabular}
\end{table*}

The results for this simulation study are summarized in Table \ref{tab:MFGvsMEKFA}.
It is shown that the accuracies for MFG based filters and MEKF are almost identical, with the differences in the attitude errors and the bias errors appearing mostly in the 3rd to 5th significant digits, if there is any.
The statistical differences found through $t$-tests are mainly caused by different measurement types, i.e. if $\delta R$ is used, MFG based filters are usually better than MEKF; conversely if $\delta\theta$ is used, MEKF is usually better.
This is natural since the measurement model for MFG based filters is formulated with $\delta R$, whereas that for MEKF is formulated with $\delta\theta$ \cite{lefferts1982kalman}, and there are some errors in matching a distribution of $\delta R$ to that of $\delta\theta$, and vice versa.
The only exception appears in $s_m=200$, where though $\delta R$ is used, the attitude error of MFG unscented filter is still greater than MEKF with the statistical significance of $p=0.035<0.05$.
Also, for $s_m=12$ and $s_m=200$, the bias error for MFG unscented filter is significantly smaller than MEKF.
But these differences are obviously negligible in practice.

\subsection{Large Initial Errors with Falsely High Confidence}

Next, the measurement noises are set as isotropic: $S_m=12I_{3\times3}$ or $\Sigma_m=0.2^2I_{3\times3}$, which is identical to the second case of the previous subsection. 
But, the initial attitude estimate is constructed by rotating the true attitude about the first body-fixed axis by 180 degrees, and the initial estimate of the bias is $[0.2,0.2,0.2]^T$ \SI{}{\radian/\second}.
However, the initial uncertainties for the attitude and bias are $S_0=200I_{3\times3}$, and $\Sigma_0=0.1^2I_{3\times3}$, respectively.
In other words, the estimator is falsely too confident about the completely wrong attitude. 
In practice, this could correspond to the scenario when large unexpected disturbance happens.

\begin{table*}
	\centering
	\caption{Attitude (\SI{}{deg}) and bias errors (\SI{}{deg/\second}) ($\pm$ sd): large initial errors with falsely high confidence \label{tab:MFGvsMEKFB}}
	\begin{tabular}{l|c|ccc}
		Measurement Model &  & MFG Analytical ($p$) & MFG Unscented ($p$) & MEKF \\ \hline \hline
		\multirow{2}{*}{\eqref{eqn:MeaMEKF} with $\sigma_m$=0.2} & attitude & 8.07$\pm$0.19 (0.006) & 8.07$\pm$0.19 (0.006) & 8.12$\pm$0.22 \\
		& bias & 6.7$\pm$1.9 ($<$0.001) & 6.7$\pm$1.9 ($<$0.001) & 8.3$\pm$2.7 \\ \hline
		\multirow{2}{*}{\eqref{eqn:MeaMFG} with $s_m$=12} & attitude & 8.23$\pm$0.22 ($<$0.001) & 8.23$\pm$0.22 ($<$0.001) & 8.30$\pm$0.26 \\
		& bias & 6.7$\pm$2.0 ($<$0.001) & 6.7$\pm$2.1 ($<$0.001) & 8.7$\pm$2.8 \\ \hline \hline
	\end{tabular}
\end{table*}

The results for this case are presented in Table \ref{tab:MFGvsMEKFB}.
It is seen that under this challenging scenario, MFG based filters are more accurate than MEKF especially in the bias estimate for both types of measurement model.
The reason for this better performance can be explained by the degree of uncertainties shown in Fig. \ref{fig:simUncertaintyB} and the estimation error given in Fig. \ref{fig:simAccuracyMagB} for a particular simulation.
Right after the attitude measurements become available, the proposed MFG estimator detects the discrepancy between the predicted attitude and the measurements caused by the completely wrong initial estimate, and it gradually increases the uncertainty in the attitude estimate to reflect the inconsistency. 
At around $t=$ \SI{0.25}{\second}, the uncertainty in attitude estimate becomes greater then the expected measurement noise of the sensor, i.e., the sensor measurement is more reliable than the current estimate. 
As such, over the measurement update step, the attitude estimate quickly shifts towards the measurements, thereby causing the attitude estimation error to be almost eliminated over a single step in Fig. \ref{fig:simAccuracyMagB}.

On the other hand, it is well known that the covariance matrix in a Kalman filter is independent of measurement errors, so its uncertainty remains low despite the large discrepancy.
As such, in the measurement update steps, MEKF consistently relies more on the incorrect prediction over the more accurate measurement, and therefore, the attitude estimation error reduces gradually. 
This further reduces the convergence rate of the bias estimation error.
In short, the MFG filter is able to adjust the uncertainty of its estimates based on the consistency of the information combined, and it is more robust to faulty initial estimates. 

\subsection{Non-isotropic Measurement Noises}

Here, we test two groups of non-isotropic measurement noises: (i) $S_m = \diag{100,0,0}$ or $\Sigma_m = \diag{10,0.01,0.01}$, and (ii) $S_m = \diag{100,50,-50}$ or $\Sigma_m = \diag{10,0.025,0.0067}$.
In each group, the distributions of $\delta R$ and $\exp(\delta\hat{\theta})$ are chosen to be similar.
In particular, the first group corresponds to the case when the rotation about the gyroscope fixed x-axis is completely unknown, and the uncertainties in the other two axes are the same; the second group corresponds to the case when even the uncertainties in the other two axes are different.
In practice, these occur when only one vector measurement is available along the first body-fixed axis, so the rotation about this vector is not measured at all.
The initial attitude, bias estimates and their distributions are chosen the same as in Section \ref{sec:simA}.

The results for this simulation study are summarized in Table \ref{tab:MFGvsMEKFC}.
It is seen that under all circumstances, the MFG filters are much more accurate than MEKF.
The reason for this better performance is that the Gaussian distribution for $\delta\theta$ is incapable of modeling the uniformly distributed rotation about the first body-fixed axis.
See Fig. \ref{fig:wrapping} for an intuitive illustration, where a Gaussian distribution defined in $\mathbb{R}$ and a wrapped Gaussian distribution \cite{mardia2009directional} defined in the circular space $\mathbb{R}/2\pi$ are compared.
It is clearly seen when the uncertainty is large, i.e. $\sigma^2=10$ which is the same as the uncertainty along the first body-fixed axis in this simulation study, the Gaussian distribution is unable to characterize the nearly uniform distribution on $[-\pi,\pi]$ properly.
Whereas the wrapped Gaussian density is much closer to a uniform distribution.
Consequently, MEKF falsely takes too much information in the rotation about the first body-fixed axis from measurements, which leads to the degraded performance. 
The proposed Bayesian estimator based on MFG is constructed on the compact manifold of $\SO$ to avoid the above issue of MEKF.

\begin{table*}
	\centering
	\caption{Attitude (\SI{}{deg}) and bias errors(\SI{}{deg/\second}) ($\pm$ sd): non-isotropic measurement noises \label{tab:MFGvsMEKFC}}
	\begin{tabular}{l|c|ccc}
		Measurement Model &  & MFG Analytical ($p$) & MFG Unscented ($p$) & MEKF \\ \hline \hline
		\multirow{2}{*}{\eqref{eqn:MeaMEKF} with $\Sigma_m = \diag{10,0.01,0.01}$} & attitude & 6.88$\pm$0.32 ($<$0.001) & 6.88$\pm$0.32 ($<$0.001) & 8.99$\pm$0.39 \\
		& bias & 3.6$\pm$1.1 ($<$0.001) & 3.6$\pm$1.1 ($<$0.001) & 4.0$\pm$1.4 \\ \hline
		\multirow{2}{*}{\eqref{eqn:MeaMFG} with $S_m = \diag{100,0,0}$} & attitude & 7.47$\pm$0.36 ($<$0.001) & 7.47$\pm$0.36 ($<$0.001) & 9.98$\pm$0.46 \\
		& bias & 3.9$\pm$1.1 ($<$0.001) & 3.9$\pm$1.1 ($<$0.001) & 4.5$\pm$1.3 \\ \hline \hline
		\multirow{2}{*}{\eqref{eqn:MeaMEKF} with $\Sigma_m = \diag{10,0.025,0.0067}$} & attitude & 6.96$\pm$0.31 ($<$0.001) & 6.96$\pm$0.31 ($<$0.001) & 9.48$\pm$0.43 \\
		& bias & 3.7$\pm$1.1 ($<$0.001) & 3.7$\pm$1.1 ($<$0.001) & 4.1$\pm$1.2 \\ \hline
		\multirow{2}{*}{\eqref{eqn:MeaMFG} with $S_m = \diag{100,50,-50}$} & attitude & 7.44$\pm$0.27 ($<$0.001) & 7.43$\pm$0.27 ($<$0.001) & 10.23$\pm$0.62 \\
		& bias & 3.5$\pm$1.0 ($<$0.001) & 3.5$\pm$1.0 ($<$0.001) & 4.3$\pm$1.2 \\ \hline \hline
	\end{tabular}
\end{table*}

\begin{figure}
	\centering
	\includegraphics{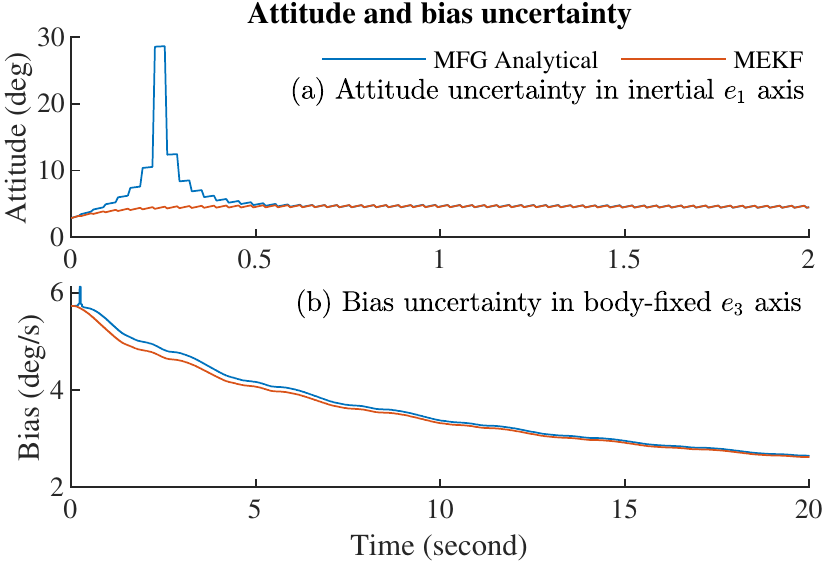}
	\caption{Attitude and bias uncertainties of a particular simulation with large initial errors and falsely high confidence.
		The attitude uncertainty is calculated as the square root of the first diagonal term of the attitude covariance matrix in the inertial frame, i.e. $R\Sigma R^T$ for MEKF and $U(\tr{S}I-S)U^T$ for the MFG filter, where $R$ is the attitude estimate and $\Sigma$ is the covariance matrix for attitude in MEKF.
		The bias uncertainty for MEKF is given as the square root of the corresponding diagonal term in the covariance matrix, and for the MFG filter is given by the square root of the third diagonal term in $\mathrm{cov}(x,x)$ calculated using \eqref{eqn:MFGEx} and \eqref{eqn:MFGExx}.
		The MFG unscented filter is similar to the MFG analytical filter, and is omitted. \label{fig:simUncertaintyB}}
\end{figure}

\section{Conclusion}

In this paper, we propose a new probability distribution, namely the matrix Fisher-Gaussian distribution (MFG), defined on the product manifold $\SO\times\real{n}$, by combining the matrix Fisher distribution and the Gaussian distribution.
MFG is constructed by conditioning a $(9+n)$-variate Gaussian distribution from the ambient Euclidean space into $\SO\times\real{n}$, with the correlation between the attitude and linear components constrained to be nonzero only in the tangent space at the mean attitude.
This construction leads to an intuitive geometric interpretation of the correlation terms of an MFG without over-parameterization.
Furthermore, a closed form approximate solution to the maximum likelihood estimation for MFG is presented by marginalizing the attitude part, which avoids using a computational costly numerical method.

Based on the proposed MFG, we design two Bayesian estimators to estimate the attitude and gyroscope bias concurrently.
The analytical filter calculates necessary moments from the gyroscope kinematic model directly, whereas the unscented filter uses sigma points to approximate these moments in an empirical fashion.
These moments are then matched to a new MFG using the MLE to propagate uncertainties.
When the attitude or vectors in the inertial frame are measured, the Bayes' formula is used to fuse the prediction with measurements, and an MFG is matched to the posterior density again using MLE.
The two filters are compared with the well-established multiplicative extended Kalman filter, where it is illustrated that their performances are almost identical for nominal cases.
However, under challenging circumstances described by (i) large initial errors with falsely high confidence, and (ii) highly non-isotropic measurement noises, the proposed Bayesian estimator with MFG exhibits substantial improvements in accuracy over MEKF. 

While this paper focuses on the attitude and gyro bias estimation, the proposed MFG can characterize the angular-linear correlation between the attitude of a rigid body and any linear random variable of an arbitrary dimension in a global fashion, which is the fundamental contribution of this paper. 
There is a great potential that the proposed MFG becomes utilized in any statistical analysis involving the rotation motion coupled with linear dynamics in the board area of science and engineering. 

\begin{figure}
	\centering
	\includegraphics{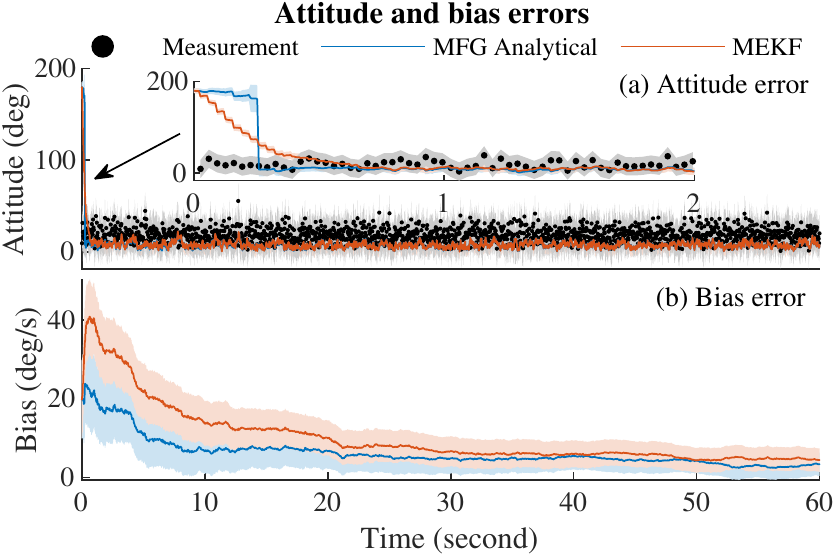}
	\caption{Attitude and bias errors in a particular simulation with large initial errors and falsely high confidence. \label{fig:simAccuracyMagB}}
\end{figure}

\begin{figure}
	\centering
	\includegraphics{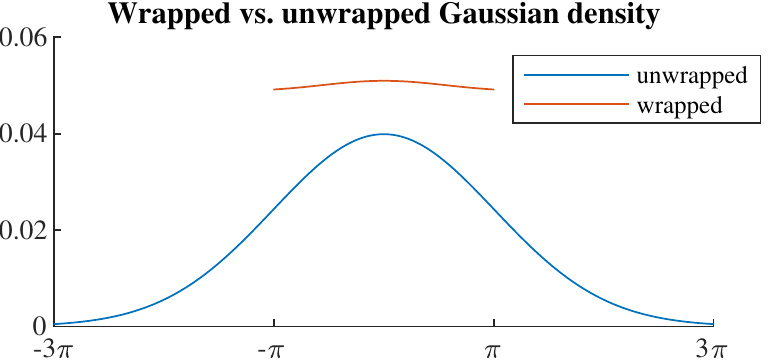}
	\caption{Gaussian density vs. wrapped Gaussian density with $\mu=0$ and $\sigma^2=10$.
	The wrapped Gaussian distribution is defined in the circular space $\mathbb{R}/2\pi$ by identifying $\theta$ with $\theta+2k\pi$ for $k\in\mathbb{Z}$, so the density beyond $[-\pi,\pi]$ is wrapped into $[-\pi,\pi]$. \label{fig:wrapping}}
\end{figure}


%

\appendices

\section{Equivalent Conditions for MFG} \label{appendix:MFGEquivalent}

In this appendix, we give a complete characterization of the equivalent conditions for MFG.
In order to first focus on the uniqueness problem with repeated singular values of $F$, we augment Definition \ref{def:psvd} by the following uniqueness condition:
\textit{the first nonzero element of each column of $U'$ is positive.} \cite{khatri1977mises}
This condition makes sure the columns of $U$ and $V$ cannot undergo simultaneous sign changes.

\begin{theorem} \label{thm:MFGEquivalent}
	Suppose $F=USV$, $\tilde{F}=\tilde{U}\tilde{S}\tilde{V}^T$ are the proper SVD of $F$ and $\tilde{F}$ with the augmented uniqueness condition.
	Then  $\mathcal{MG}(\mu,\Sigma,P,U,S,V)$ and $\mathcal{MG}(\tilde{\mu},\tilde{\Sigma},\tilde{P},\tilde{U},\tilde{S},\tilde{V})$, using either Definition \ref{def:MFG1} or \ref{def:MFG2}, are equivalent if and only if $\mu=\tilde{\mu}$, $S=\tilde{S}$, and one of the following conditions is satisfied:
	\begin{enumerate}
		\item if $s_1=s_2=s_3=0$, then $\Sigma = \tilde{\Sigma}$. \label{case:s1=s2=s3=0}
		\item if $s_1 \neq s_2=s_3=0$, then $\exists \theta_1, \theta_2 \in \mathbb{R}$ such that $\tilde{U}=UT_1$, $\tilde{V}=VT_2$ where $T_1 = \exp(\theta_1\hat{e}_1)$ and $T_2 = \exp(\theta_2\hat{e}_1)$, $[\tilde{P}_{:,2},\tilde{P}_{:,3}] = [P_{:,2},P_{:,3}]\begin{bmatrix} \cos\theta_1 & -\sin\theta_1 \\ \sin\theta_1 & \cos\theta_1\end{bmatrix}$ where $P_{:,i}$ is the $i$-th column of $P$, and $\Sigma = \tilde{\Sigma}$. \label{case:s1 neq s2=s3=0}
		\item if $s_1=s_2=s_3 \neq 0$, then $\exists T\in\SO$ such that $\tilde{U}=UT$, $\tilde{V}=VT$, $\tilde{P}=PT$ and $\Sigma=\tilde{\Sigma}$. \label{case:s1=s2=s3 neq 0}
		\item if $s_1 \neq s_2=s_3 \neq 0$, then $\exists \theta\in\mathbb{R}$ such that $\tilde{U}=UT$, $\tilde{V}=VT$, $\tilde{P}=PT$, where $T=\exp(\theta\hat{e}_1)$, and $\Sigma=\tilde{\Sigma}$.\label{case:s1 neq s2=s3 neq 0}
		\item if $s_1=s_2 \neq |s_3|$, then $\exists \theta\in\mathbb{R}$ such that $\tilde{U}=UT$, $\tilde{V}=VT$, $\tilde{P}=PT$, where $T=\exp(\theta\hat{e}_3)$, and $\Sigma=\tilde{\Sigma}$.\label{case:s1=s2 neq |s3|}
		\item if $s_1 \neq s_2 \neq |s_3|$, then $U=\tilde{U}$, $V=\tilde{V}$, $P=\tilde{P}$, and $\Sigma=\tilde{\Sigma}$. \label{case:s1 neq s2 neq |s3|}
		\item if using Definition \ref{def:MFG1} and $s_1=s_2=-s_3 \neq 0$, then $\exists T\in\SO$ such that $\tilde{U}=UT$, $\tilde{V}=V\mathcal{D}_{12}T\mathcal{D}_{12}$, $\tilde{P}=PT$, and $\tilde{\Sigma} = \Sigma + P\left[ T(\tr{S}I-S)T^T - (\tr{S}I-S) \right]P^T$. \label{case:s1=s2=-s3 neq 0}
		\item if using Definition \ref{def:MFG1} and $s_1 \neq s_2=-s_3 \neq 0$, then $\exists \theta\in\mathbb{R}$ such that $\tilde{U}=UT$, $\tilde{V}=VT^T$, $\tilde{P}=PT$, and $\tilde{\Sigma} = \Sigma + P\left[ T(\tr{S}I-S)T^T - (\tr{S}I-S) \right]P^T$, where $T=\exp(\theta\hat{e}_1)$. \label{case:s1 neq s2=-s3 neq 0}
	\end{enumerate}
\end{theorem}
\begin{proof}
	By Proposition \ref{prop:MFGEquivalent}, the two MFGs are equivalent if and only if $F=\tilde{F}$, $\mu_c=\tilde{\mu}_c$ for all $R\in\SO$, and $\Sigma=\tilde{\Sigma}_c$.

    The conditions provided in the theorem begin with $S=\tilde {S}$. 
    We first consider additional conditions to make them equivalent to $F=\tilde{F}$.
	Since a matrix cannot have two sets of different singular values, $F=\tilde{F}$ implies $S=\tilde{S}$.
    For case \ref{case:s1=s2=s3=0}), $S=\tilde{S}$ trivially implies $F=\tilde {F}$ as they are all zeros.
	For case \ref{case:s1 neq s2 neq |s3|}), $F=\tilde{F}$ if and only if their proper SVDs are the same, due to the augmented uniqueness condition for SVD when the singular values are different.
    In short, $F=\tilde{F}$ if and only if $S=\tilde{S}$ for cases 1) and 6).
	However, when $F$ has repeated singular values, $U$ and $V$ are only unique up to a rotation.
	We consider three possibilities regarding the multiplicity of $S$.
	
	(i) If $s_1 \neq s_2=s_3=0$, corresponding to case \ref{case:s1 neq s2=s3=0}) in the theorem, then $F=\tilde{F}$ if and only if $U\diag{s_1,0,0}V^T = \tilde{U}\diag{s_1,0,0}\tilde{V}^T$.
	This means the first column of $U$ and $V$ is unique, while other columns can be arbitrarily chosen as long as $U,V\in\SO$, which is the same as the conditions on $U,V$ in case \ref{case:s1 neq s2=s3=0}).

	(ii) If $s_1=s_2$ and(or) $s_2=s_3 \neq 0$, which corresponds to case \ref{case:s1=s2=s3 neq 0}) to \ref{case:s1=s2 neq |s3|}) in the theorem, then the left and right singular vectors (columns of $U$ and $V$) w.r.t the repeated singular value form subspaces of $\real{3}$.
	And $F=\tilde{F}$ if and only if the left and right singular vectors w.r.t the repeated singular value are rotated by the same rotation matrix, while the singular vectors w.r.t the non-repeated singular value are the same.
	The above condition on $U,V$ is the same as those given in case \ref{case:s1=s2=s3 neq 0}) to \ref{case:s1=s2 neq |s3|}).
	
	(iii) If $s_2=-s_3 \neq 0$, corresponding to cases \ref{case:s1=s2=-s3 neq 0}) and \ref{case:s1 neq s2=-s3 neq 0}) in the theorem.
	Let $\mathcal{D}_{12} = \diag{1,1,-1}$, $S' = S\mathcal{D}_{12} \triangleq \diag{s'_1,s'_2,s'_3}$. 
    Then $USV^T = US'\mathcal{D}_{12}V^T$.
	Since $s'_2=s'_3$, for all $\theta\in\mathbb{R}$ and $T=\exp(\theta\hat{e}_1)$, we have
	\begin{align*}
		US'\mathcal{D}_{12}V^T &= UTS'T^T\mathcal{D}_{12}V^T = UTS'\mathcal{D}_{12}\mathcal{D}_{12}T^T\mathcal{D}_{12}V^T \\
		&= UTS(V\mathcal{D}_{12}T\mathcal{D}_{12})^T = UTS(VT^T)^T.
	\end{align*}
	This shows the sufficiency and necessity for the conditions on $U,V$ in case \ref{case:s1 neq s2=-s3 neq 0}).
	For case \ref{case:s1=s2=-s3 neq 0}), $T$ can be arbitrarily chosen from $\SO$, but in general $\mathcal{D}_{12}T\mathcal{D}_{12} \neq T^T$.
	
	Next, we consider the equivalent conditions for $\mu_c=\tilde{\mu}_c$ and $\Sigma_c=\tilde{\Sigma}_c$.
	For case \ref{case:s1=s2=s3=0}), the matrix Fisher density parts reduce to one, $\mu_c = \mu$, and $\Sigma_c=\Sigma$, so the remaining Gaussian density parts are equivalent if and only if their means and covariance matrices are the same, as shown in Proposition \ref{prop:MFGEquivalent}.
	Next, the three possibilities regarding the multiplicity of $S$ are considered individually as follows.
	
	For (i), let $Q=U^TRV$, since $\tilde{U}=UT_1$, $\tilde{V}=VT_2$ and $S=\tilde{S}$, we have
	\begin{align} \label{eqn:muc'=muc}
		\tilde{\mu}_c &= \tilde{\mu} + \tilde{P}(T_1^TU^TRVT_2S-ST_2^TV^TR^TUT_1)^\vee \nonumber \\
		&= \tilde{\mu} + \tilde{P}T_1^T(U^TRVT_2ST_1^T-T_1ST_2^TV^TR^TU)^\vee \nonumber \\
		&= \tilde{\mu} + \tilde{P}T_1^T(QS-SQ^T)^\vee \nonumber \\
		&= \tilde{\mu} + \tilde{P}T_1^T\begin{bmatrix} 0 & -s_1Q_{31} & s_1Q_{21} \end{bmatrix}^T.
	\end{align}
	Note that $Q$ and $R$ have a one-to-one correspondence, so $\mu_c = \tilde{\mu}_c$ for all $R\in\SO$ is equivalent to $\mu_c = \tilde{\mu}_c$ for all $Q\in\SO$.
	It is clear the conditions on $\mu$ and $P$ in case \ref{case:s1 neq s2=s3=0}) are sufficient for this.
	On the other hand, if $\mu_c = \tilde{\mu}_c$ for all $Q\in\SO$, substituting $Q=I_{3 \times 3}$ into \eqref{eqn:muc'=muc} proves $\mu=\tilde{\mu}$; substituting $Q = \begin{bmatrix} 0 & 0 & 1 \\ 0 & 1 & 0 \\ -1 & 0 & 0 \end{bmatrix} \triangleq Q_2$ and $Q = \begin{bmatrix} 0 & -1 & 0 \\ 1 & 0 & 0 \\ 0 & 0 & 1 \end{bmatrix} \triangleq Q_3$ into \eqref{eqn:muc'=muc} proves $[\tilde{P}_{:,2},\tilde{P}_{:,3}] = [P_{:,2},P_{:,3}]\begin{bmatrix} \cos\theta_1 & -\sin\theta_1 \\ \sin\theta_1 & \cos\theta_1\end{bmatrix}$.
	Finally, denote $T'_1 = \begin{bmatrix} \cos\theta_1 & -\sin\theta_1 \\ \sin\theta_1 & \cos\theta_1\end{bmatrix}$, then
	\begin{align}
		&\tilde{P}(\tr{S}I_{3 \times 3}-S)(\tilde{P})^T \nonumber \\
		= &\begin{bmatrix} P_{:,1} & \begin{bmatrix} P_{:,2} & P_{:,3} \end{bmatrix}T'_1 \end{bmatrix} \begin{bmatrix} 0 & 0 & 0 \\ 0 & s_1 & 0 \\ 0 & 0 & s_1 \end{bmatrix} \begin{bmatrix} P_{:,1}^T \\ (T'_1)^T \begin{bmatrix} P_{:,2}^T \\ P_{:,3}^T \end{bmatrix} \end{bmatrix} \nonumber \\
		= &\begin{bmatrix} P_{:,2} & P_{:,3} \end{bmatrix}T'_1 \begin{bmatrix} s_1 & 0 \\ 0 & s_1 \end{bmatrix} (T'_1)^T \begin{bmatrix} P_{:,2}^T \\ P_{:,3}^T \end{bmatrix} \nonumber \\
		= &P(\tr{S}I_{3 \times 3}-S)P^T,
	\end{align}
	so $\Sigma_c = \tilde{\Sigma}_c$ if and only if $\Sigma = \tilde{\Sigma}$.
	
	For (ii), a similar calculation as in \eqref{eqn:muc'=muc} shows $\tilde{\mu}_c = \tilde{\mu} + \tilde{P}T^T \begin{bmatrix} s_2Q_{32}-s_3Q_{23} & s_3Q_{13}-s_1Q_{31} & s_1Q_{21}-s_2Q_{12} \end{bmatrix}^T$.
	The sufficiency of the conditions on $\mu$ and $P$ in case \ref{case:s1=s2=s3 neq 0}) to \ref{case:s1=s2 neq |s3|}) follows immediately.
	For the necessary direction, substituting $Q = I_{3 \times 3}$ into the above equation proves $\mu = \tilde{\mu}$;
	substituting $Q = \begin{bmatrix} 1 & 0 & 0 \\ 0 & 0 & -1 \\ 0 & 1 & 0 \end{bmatrix} \triangleq Q_1$, $Q=Q_2$ and $Q=Q_3$ proves $\tilde{P} = PT$ since $s_i+s_j>0$ (or $<0$ if $s_3<0$ in Definition \ref{def:MFG2}) for any $i \neq j$ in case \ref{case:s1=s2=s3 neq 0}) to \ref{case:s1=s2 neq |s3|}).
	Finally, note that the repeated diagonal entries of $S$ and $\tr{S}I_{3 \times 3}-S$ share the same indices, thus $T(\tr{S}I_{3 \times 3}-S)T^T = \tr{S}I_{3 \times 3}-S$.
	This proves $\Sigma_c = \tilde{\Sigma}_c$ if and only if $\Sigma = \tilde{\Sigma}$.
	The same argument in this paragraph also proves case \ref{case:s1 neq s2 neq |s3|}) in the theorem by letting $T = I_{3 \times 3}$.
	
	For (iii), a similar calculation as in \eqref{eqn:muc'=muc} yields
	\begin{align*}
		\tilde{\mu}_c &= \tilde{\mu} + \tilde{P}T^T(Q\mathcal{D}_{12}S-S\mathcal{D}_{12}Q^T)^\vee \nonumber \\
		&= \tilde{\mu} + \tilde{P}T^T\begin{bmatrix} s_2Q_{32}+s_3Q_{23} \\ -s_3Q_{13}-s_1Q_{31} \\ s_1Q_{21}-s_2Q_{12} \end{bmatrix}^T,
	\end{align*}
	which is exactly the same equation as in (ii) by replacing $s_3$ with $-s_3$.
	Therefore, the sufficiency and necessity of the conditions on $\mu$ and $P$ in case \ref{case:s1=s2=-s3 neq 0}) and \ref{case:s1 neq s2=-s3 neq 0}) follow from the same argument in (ii).
	Finally, we have
	\begin{equation}
		\tilde{\Sigma}_c = \tilde{\Sigma} + PT(\tr{S}I-S)T^TP^T.
	\end{equation}
	However, different from (ii), $T(\tr{S}I_{3 \times 3}-S)T^T \neq \tr{S}I_{3 \times 3}-S$.
	And as a result, $\Sigma_c = \tilde{\Sigma}_c$ if and only if $\tilde{\Sigma} = \Sigma + P\left[ T(\tr{S}I-S)T^T - (\tr{S}I-S) \right]P^T$.
\end{proof}

In other words, Theorem \ref{thm:MFGEquivalent} states that if $F$ has repeated singular values and $s_3\geq 0$, then an MFG can be parameterized differently by rotating $U$, $V$, and $P$ in a consistent way.
A notable difference occurs when $s_3<0$ between the two different definitions of MFG.
For definition \ref{def:MFG2}, this special case is included in case \ref{case:s1=s2=s3 neq 0}) and \ref{case:s1 neq s2=s3 neq 0}), and is the same as when $s_3>0$.
On the other hand, for Definition \ref{def:MFG1}, it is very different from when $s_3>0$ in the way that $\tilde{\Sigma} \neq \Sigma$ as indicated in case \ref{case:s1=s2=-s3 neq 0}) and \ref{case:s1 neq s2=-s3 neq 0}).
The reason for this is when $s_2=-s_3$, $T(\tr{S}I-S)T^T \neq \tr{S}I-S$.


The uniqueness problem caused by simultaneous sign changes of any two columns of $U$ and $V$ is dealt with in the following proposition.
\begin{proposition}
	Let $i\in\{1,2,3\}$, then $\mathcal{MG}(\mu,\Sigma,P,U,S,V)$ and $\mathcal{MG}(\mu,\Sigma,P\mathcal{D}_i,U\mathcal{D}_i,S,V\mathcal{D}_i)$ are equivalent.
\end{proposition}
\begin{proof}
	This is evident by Proposition \ref{prop:MFGEquivalent} and the calculations in the proof of Theorem \ref{thm:MFGEquivalent}.
\end{proof}

\section{Calculations for Propagated Moments} \label{appendix:moments}

In this appendix, calculations for the moments involving $(R_{k+1},x_{k+1})$ in Section \ref{sec:analyticalProp} are given.

\begin{proof}[Proof of Theorem \ref{thm:ER_{k+1}}]
	First, expand the first exponential term in \eqref{eqn:R_{k+1}factorization} into an infinite sum as
	\begin{align*}
		&e^{h (\hat x_k-\hat{\mu}_k) + (H_u\Delta W_u)^\wedge + o(h)} \\
        & = \sum_{i=0}^\infty \frac{1}{i!}\left\{ h (\hat x_k-\hat{\mu}_k) + (H_u\Delta W_u)^\wedge + o(h) \right\}^i.
	\end{align*}
	Note that (i) $\Delta W_u$ is a zero mean Gaussian vector with covariance matrix $hI_{3\times3}$, so its odd order moments are zero, and $\expect{\left((H_u\Delta W_u)^\wedge\right)^{2n}} \sim O(h^n)$; (ii) $o(h)$ in the above equation only has terms of order at least $h^2$ or $h\Delta W_u$ as shown in \eqref{eqn:R_{k+1}factorizationProof}.
	Combining these two observations, we have the first order approximation of $\expect{R_{k+1}}$ as
	\begin{align} \label{eqn:E(R_{k+1})Expand}
		&\expect{R_{k+1}} = \big\{ \expect{R_k} + h \expect{R_k(\hat x_k-\hat{\mu}_k)} \nonumber \\
        & + \frac{1}{2}\expect{R_k((H_u\Delta W_u)^\wedge)^2} \big\} e^{h(\hat{\Omega}_k+\hat{\mu}_k)} + O(h^2).
	\end{align}
    Then, since $(R_k,x_k)$ follows MFG,
	\begin{align*}
        \expect{R_k(\hat x_k-\hat{\mu}_k)} = \expect{R_k\widehat{P\nu_{R_k}}} = U_k\expect{Q_kV_k^T\widehat{P_k\nu_{R_k}}}.
	\end{align*}
	In addition, due to the independence of $R_k$ and $\Delta W_u$,
	\begin{align*}
		\expect{R_k((H_u\Delta W_u)^\wedge)^2} = \expect{R_k}(G_u-\tr{G_u}I).
	\end{align*}
    Substituting the above two equations into \eqref{eqn:E(R_{k+1})Expand} yields \eqref{eqn:E(R_{k+1})}.
\end{proof}

\begin{proof}[Proof of Theorem \ref{thm:otherMoments}]
    Equation \eqref{eqn:Ex_{k+1}} and \eqref{eqn:Exx_{k+1}} follow from the independence of $x_k$ and $\Delta W_v$, and the linearity of expectation.
    Equation \eqref{eqn:EvR_{k+1}} is from the fact that $\expect{Q_{k+1}} = U_{k+1}^T\expect{R_{k+1}}V_{k+1}$ and $S_{k+1}$ are both diagonal.

    In \eqref{eqn:R_{k+1}factorization}, we denote the exponent of the stochastic part as $\xi \triangleq h(x_k-\mu_k) + H_u\Delta W_u\in\real{3}$, and the deterministic part as $\delta R_k \triangleq e^{h(\hat{\Omega}_k+\hat{\mu}_k)}\in\SO$ such that
    \begin{align*}
        R_{k+1} = R_k e^{\xi + o(h)} \delta R_k.
    \end{align*}
    For any $A \in \real{3 \times 3}$, let $\Lambda(A)\in\real{3}$ be defined as $\Lambda(A) = (U_{k+1}^TAV_{k+1}S_{k+1}-S_{k+1}V_{k+1}^TA^TU_{k+1})^\vee$. 
    For example, $\nu_{R_{k+1}} = \Lambda(R_{k+1})$.
    Further, $\Delta W_v$ in \eqref{eqn:DistSDEBias} is independent of $R_k$ and $\xi$.
    Therefore, 
    \begin{align*}
        \expect{x_{k+1}\nu^T_{R_{k+1}}} = \sum_{i=0}^\infty \frac{1}{i!} \expect{x_k \Lambda (R_k(\hat{\xi}+o(h))^i \delta R_k)^T},
	\end{align*}
    after expanding $e^{\xi + o(h)}$. 
    Similar with \eqref{eqn:E(R_{k+1})Expand},  the first order approximation is 
    \begin{align} \label{eqn:ExvR_{k+1}Expand}
        & \expect{x_{k+1}\nu^T_{R_{k+1}}}  = \expect{x_k \Lambda(R_k\delta R_k )^T} \nonumber\\
        & \quad + h\expect{x_k \Lambda( R_k(\hat x_k-\hat{\mu}_k)\delta R_k )^T} \nonumber \\
        & \quad + \frac{1}{2}\expect{x_k \Lambda( R_k((H_u\Delta W_u)^\wedge)^2\delta R_k )^T} + O(h^2).
    \end{align}
	Now we calculate the three terms on the right hand side of \eqref{eqn:ExvR_{k+1}Expand}.
    Similar with \eqref{eqn:ExxTranspose}, the first term is
	\begin{align*}
        \expect{x_k\Lambda( R_k\delta R_k )^T} & = \mu_k\expect{\Lambda( R_k\delta R_k )}^T \\
                                               & \quad + P_k \expect{\nu_{R_k} \Lambda(R_k\delta R_k)^T },
	\end{align*}
	where
	\begin{align*}
        &\expect{\Lambda( R_k\delta R_k )} = \tilde{U} (\expect{Q_k}\tilde{S}^T-\tilde{S}\expect{Q_k}^T)^\vee = \tilde{U}\expect{\tilde{\nu}_R}, \\
        &\expect{\nu_{R_k} \Lambda( R_k\delta R )^T} \\
        &\quad= \expect{(Q_k S_k-S_kQ_k^T)^\vee ( (Q_k\tilde{S}^T-\tilde{S}Q_k^T)^\vee )^T} \tilde{U}^T \\
                                                      & \quad = \expect{\nu_{R_k} \tilde{\nu}^T_R}\tilde{U}^T.
	\end{align*}
    Next, we have
	\begin{align*}
        \Lambda & ( R_k(\hat{x}_k-\hat{\mu}_k)\delta R_k ) \\
                &= \tilde{U}\left( \mathrm{tr}(Q_k\tilde{S}^T)I - Q_k\tilde{S}^T \right)Q_kV_k^T(x_k-\mu_k) \\
		&= \tilde{U}\Gamma_QV_k^T(x_k-\mu_k).
	\end{align*}
    Therefore, the second term of \eqref{eqn:ExvR_{k+1}Expand} is
	\begin{align*}
		&\expect{x_k \Lambda ( R_k(\hat{x}_k-\hat{\mu}_k)\delta R_k )^T} = \expect{x_k(x_k-\mu_k)^TV_k\Gamma_Q^T}\tilde{U}^T \\
        &\qquad = \expect{(\Sigma_{c_k} + \mu_k\nu_{R_k}^TP_k^T + P_k\nu_{R_k}\nu_{R_k}^TP_k^T)V_k \Gamma_Q^T}\tilde{U}^T.
	\end{align*}
    Due to the independence of $\Delta W_u$ with other random variables, the third term of \eqref{eqn:ExvR_{k+1}Expand} is given by
	\begin{align*}
		&\expect{x_k \Lambda( R_k((H_u\Delta W_u)^\wedge)^2\delta R_k )^T} \\
		&\quad = h\expect{x_k\Lambda( R_kG_u\delta R_k )^T} - h\tr{G_u}\expect{x_k\Lambda( R_k\delta R_k )^T}.
	\end{align*}
    Let $\tilde{\tilde{V}} = V_{k+1}^T\delta R_k^T G_u^TV_k$ and $\tilde{\tilde{S}} = \tilde{U}^TS_k\tilde{\tilde{V}}$. 
    The first term of the above is
	\begin{align*}
		&\expect{x_k \Lambda( R_kG_u\delta R_k )^T} = \mu\left( (\expect{Q_k}\tilde{\tilde{S}}^T-\tilde{\tilde{S}}\expect{Q_k}^T)^\vee \right)^T \tilde{U}^T \\
        &\quad + P_k\expect{ \nu_{R_k} \left( (Q\tilde{\tilde{S}}^T-\tilde{\tilde{S}}Q^T)^\vee \right)^T}\tilde{U}^T \\
		&\quad = \mu\expect{\tilde{\tilde{\nu}}^T_R}\tilde{U}^T + P\expect{\nu_R\tilde{\tilde{\nu}}^T_R}\tilde{U}^T.
	\end{align*}
	Substituting these three terms into \eqref{eqn:ExvR_{k+1}Expand}, equation \eqref{eqn:ExvR_{k+1}} is derived.
	
	Finally, we present the expression for \eqref{eqn:EvRvR_{k+1}}.
	Similar to the derivation of \eqref{eqn:ExvR_{k+1}Expand}, after expanding the exponential term, we obtain
	\begin{align*}   
        \expect{v_{R_{k+1}}v^T_{R_{k+1}}} & = \sum_{i,j=0}^\infty \frac{1}{i! j!} \expect{\Lambda( R_k(\hat{\xi}+o(h))^i \delta R_k )  \right. \nonumber\\
        & \qquad \left. \times \Lambda( R_k(\hat{\xi}+o(h))^j \delta R_k )^T}.
    \end{align*}
    After collecting the first order terms, 
    \begin{align}\label{eqn:EvRvR_{k+1}Expand}
        &\expect{v_{R_{k+1}}v^T_{R_{k+1}}}  = \expect{ \Lambda( R_k\delta R )\Lambda( R_k\delta R )^T } \nonumber\\
        &\quad + h\expect{\Lambda( R_k\delta R )\, \Lambda( R_k(\hat{x}_k-\hat{\mu}_k)\delta R_k )^T} \nonumber \\
        &\quad + h\expect{\Lambda( R_k\delta R )\, \Lambda( R_k(\hat{x}_k-\hat{\mu}_k)\delta R_k )^T}^T \nonumber \\
		&\quad + \frac{1}{2}\expect{\Lambda( R_k\delta R )\, \Lambda( R_k((H_u\Delta W_u)^\wedge)^2\delta R_k )^T} \nonumber \\
		&\quad + \frac{1}{2}\expect{\Lambda( R_k\delta R )\, \Lambda( R_k((H_u\Delta W_u)^\wedge)^2\delta R_k )^T}^T \nonumber \\
		&\quad + \expect{\Lambda( R_k(H_u\Delta W_u)^\wedge\delta R_k )\, \Lambda( R_k(H_u\Delta W_u)^\wedge\delta R_k )^T }\nonumber\\
        &\quad + O(h^2), 
	\end{align}
    where the first term on the right hand side is
	\begin{align*}
		&\expect{\Lambda( R_k\delta R )\, \Lambda( R_k\delta R )^T} \\
		&\quad = \tilde{U} \expect{(Q_k\tilde{S}^T-\tilde{S}Q_k)^\vee((Q_k \tilde{S}^T-\tilde{S}Q_k)^\vee)^T} \tilde{U}^T \\
		&\quad = \tilde{U}\expect{\tilde{\nu}_R\tilde{\nu}^T_R}\tilde{U}^T,
	\end{align*}
	and the second term is 
	\begin{align*}
		&\expect{\Lambda( R_k\delta R_k )\, \Lambda( R_k(\hat{x}_k-\hat{\mu}_k )\delta R_k )^T} \\
        &\quad = \tilde{U} \expect{(Q_k \tilde{S}^T-\tilde{S}Q_k^T)^\vee \nu^T_{R_k} P_k^T V_k \Gamma_Q^T} \tilde{U}^T \\
        &\quad = \tilde{U} \expect{\tilde{\nu}_R\nu_{R_k}^T P^T V_k \Gamma_Q^T} \tilde{U}^T,
	\end{align*}
    Similarly, the fourth term on the right hand side of \eqref{eqn:EvRvR_{k+1}Expand} is
	\begin{align*}
		&\expect{\Lambda( R_k\delta R_k ) \Lambda( R_k((H_u\Delta W_u)^\wedge)^2\delta R_k )^T } \\ 
		&\quad = h\tilde{U}\expect{\tilde{\nu}_R\tilde{\tilde{\nu}}^T_R}\tilde{U}^T - \tr{G_u}h\tilde{U}\expect{\tilde{\nu}_R\tilde{\nu}_R^T}\tilde{U}^T,
	\end{align*}
    and the sixth term is
	\begin{align*}
		&\expect{\Lambda( R_k(H_u\Delta W_u)^\wedge\delta R_k ) \Lambda( R_k(H_u\Delta W_u)^\wedge\delta R_k )^T} \\
		&\quad = h\tilde{U} \expect{\Gamma_Q V_k^T G_u V_k \Gamma_Q^T} \tilde{U}^T.
	\end{align*}
	Substituting these into \eqref{eqn:EvRvR_{k+1}Expand} yields \eqref{eqn:EvRvR_{k+1}}.
\end{proof}

In order to calculate the moments on the right hand side of \eqref{eqn:ExvR_{k+1}} and \eqref{eqn:EvRvR_{k+1}} involving $\Gamma_Q$ using the moments of $Q_k$ up to the third order, we give a formula for $\Gamma_Q$ that depends linearly on $Q_k$.

\begin{lemma} \label{lemma:GammaQSimplified}
	$\Gamma_Q$ can be simplified as follows:
	\begin{align}
		(\Gamma_Q)_{11} &= \tilde{S}_{22}Q_{33} + \tilde{S}_{33}Q_{22} - \tilde{S}_{23}Q_{32} - \tilde{S}_{32}Q_{23} \nonumber \\
		(\Gamma_Q)_{12} &= \tilde{S}_{23}Q_{31} + \tilde{S}_{31}Q_{23} - \tilde{S}_{21}Q_{33} - \tilde{S}_{33}Q_{21} \nonumber \\
		(\Gamma_Q)_{13} &= \tilde{S}_{21}Q_{32} + \tilde{S}_{32}Q_{21} - \tilde{S}_{22}Q_{31} - \tilde{S}_{31}Q_{22} \nonumber \\
		(\Gamma_Q)_{21} &= \tilde{S}_{13}Q_{32} + \tilde{S}_{32}Q_{13} - \tilde{S}_{12}Q_{33} - \tilde{S}_{33}Q_{12} \nonumber
	\end{align}
	\begin{align}
		(\Gamma_Q)_{22} &= \tilde{S}_{11}Q_{33} + \tilde{S}_{33}Q_{11} - \tilde{S}_{13}Q_{31} - \tilde{S}_{31}Q_{13} \nonumber \\
		(\Gamma_Q)_{23} &= \tilde{S}_{12}Q_{31} + \tilde{S}_{31}Q_{12} - \tilde{S}_{11}Q_{32} - \tilde{S}_{32}Q_{11} \nonumber \\
		(\Gamma_Q)_{31} &= \tilde{S}_{12}Q_{23} + \tilde{S}_{23}Q_{12} - \tilde{S}_{13}Q_{22} - \tilde{S}_{22}Q_{13} \nonumber \\
		(\Gamma_Q)_{32} &= \tilde{S}_{13}Q_{21} + \tilde{S}_{21}Q_{13} - \tilde{S}_{11}Q_{23} - \tilde{S}_{23}Q_{11} \nonumber \\
		(\Gamma_Q)_{33} &= \tilde{S}_{11}Q_{22} + \tilde{S}_{22}Q_{11} - \tilde{S}_{12}Q_{21} - \tilde{S}_{21}Q_{12}.
	\end{align}
\end{lemma}
\begin{proof}
	This is directly calculated using the fact that for any $Q\in\SO$, $Q_{ij} = (Q^{-1})_{ji} =(-1)^{i+j}M_{ij}$, where $M_{ij}$ is the minor of $Q_{ij}$.
\end{proof}




\ifCLASSOPTIONcaptionsoff
  \newpage
\fi



\bibliographystyle{IEEEtran}
\bibliography{TACMFG}
%



%

\begin{IEEEbiography}{Weixin Wang}
Biography text here.
\end{IEEEbiography}

\begin{IEEEbiographynophoto}{Taeyoung Lee}
Biography text here.
\end{IEEEbiographynophoto}





\end{document}